\documentclass[12pt]{article}
\usepackage{amsfonts}
\usepackage{amsthm}
\usepackage{latexsym,amssymb}
\usepackage{mathrsfs}
\usepackage{graphicx}
\usepackage{authblk}
\usepackage{psfrag}
\usepackage{amsbsy}
\usepackage{amsopn, amstext}
\usepackage{amsmath,multicol,amsopn}
\usepackage{color}
\usepackage{subfig}
\usepackage{float}

\def\sqr#1#2{{\vcenter{\vbox{\hrule height.#2pt
              \hbox{\vrule width.#2pt height#1pt \kern#1pt \vrule width.#2pt}
              \hrule height.#2pt}}}}

\def\dbR{{\mathop{\rm l\negthinspace R}}}

\def\3n{\negthinspace \negthinspace \negthinspace }
\def\2n{\negthinspace \negthinspace }
\def\1n{\negthinspace }

\def\dbF{{\mathbb{F}}}

\def\dbP{{\mathbb{P}}}

\def\dbR{{\mathbb{R}}}

\def\ds{\displaystyle}

\def\={\buildrel \triangle \over =}

\def\a{\alpha}

\def\g{\gamma}

\def\e{\varepsilon}

\def\l{\lambda}

\def\t{\times}
\def\f{\varphi}
\def\th{\theta}
\def\om{\omega}

\def\ns{\noalign{\ss} }

\def\pa{\partial}

\def\G{\Gamma}
\def\D{\Delta}

\def\Si{\Sigma}

\def\Om{\Omega}

\def\cA{{\cal A}}

\def\cE{{\cal E}}

\def\cH{{\cal H}}

\def\cJ{{\cal J}}

\def\mE{{\mathbb{E}}}

\def\no{\noindent}

\def\ms{\medskip}
\def\bs{\bigskip}
\def\q{\quad}
\def\qq{\qquad}

\def\lan{\big\langle}
\def\ran{\big\rangle}

\def\max{\mathop{\rm max}}
\def\min{\mathop{\rm min}}

\def\pa{\partial}

\def\cd{\cdot}
\def\cds{\cdots}

\def\div{\hbox{\rm div$\,$}}

\def\deq{\mathop{\buildrel\D\over=}}

\def\({\Big (}
\def\){\Big )}
\def\[{\Big[}
\def\]{\Big]}
\def\be{\begin{equation}}
\def\bel{\begin{equation}\label}
\def\ee{\end{equation}}
\def\bt{\begin{theorem}}
\def\bcd{\begin{condition}}
\def\ecd{\end{condition}}
\def\et{\end{theorem}}
\def\bc{\begin{corollary}}
\def\ec{\end{corollary}}
\def\bde{\begin{definition}}
\def\ede{\end{definition}}
\def\bl{\begin{lemma}}
\def\el{\end{lemma}}
\def\bp{\begin{proposition}}
\def\ep{\end{proposition}}
\def\br{\begin{remark}}
\def\er{\end{remark}}
\def\ba{\begin{array}}
\def\ea{\end{array}}
\def\ed{\end{document}}
\def\ns{\noalign{\ms}}
\def\ds{\displaystyle}

\def\square#1{\vbox{\hrule\hbox{\vrule height#1%
     \kern#1\vrule}\hrule}}
\def\rectangle#1#2{\vbox{\hrule\hbox{\vrule height#1%
     \kern#2\vrule}\hrule}}

\font\tenbb=msbm10 \font\sevenbb=msbm7
\font\fivebb=msbm5

\newfam\bbfam
\scriptscriptfont\bbfam=\fivebb
\textfont\bbfam=\tenbb
\scriptfont\bbfam=\sevenbb

\newtheorem{lemma}{Lemma}[section]
\newtheorem{remark}{Remark}[section]
\newtheorem{example}{Example}[section]
\newtheorem{theorem}{Theorem}[section]
\newtheorem{corollary}{Corollary}[section]

\newtheorem{definition}{Definition}[section]
\newtheorem{proposition}{Proposition}[section]
\newtheorem{condition}{Condition}[section]

\allowdisplaybreaks

\makeatletter
   
   \@addtoreset{equation}{section}
\makeatother

\begin{document}
\title{\bf An inverse Cauchy problem of a stochastic hyperbolic equation  \ms}

\author[]{Fangfang Dou\thanks{Corrsponding author. Email: fangfdou@uestc.edu.cn.} }
\author[]{Peimin L\"{u}}
\affil{\small\it School of Mathematical Sciences, University of Electronic Science and Technology of China, Chengdu, China}
\date{}

\maketitle

\begin{abstract}
In this paper, we investigate an  inverse Cauchy problem for a stochastic hyperbolic equation. A Lipschitz type observability estimate is established using a pointwise Carleman identity. By minimizing the constructed Tikhonov-type functional, we obtain a regularized approximation to the problem. The properties of the approximation are studied by means of the Carleman estimate and Riesz representation theorem. Leveraging kernel-based learning theory, we simulate numerical algorithms based on the proposed regularization method. These reconstruction algorithms are implemented and validated through several numerical experiments, demonstrating their feasibility and accuracy.
\end{abstract}

\bs

\no{\bf 2010 Mathematics Subject
Classification}. 65N21, 60H15, 65D12.

\bs

\no{\bf Key Words}. Stochastic hyperbolic equations,  inverse Cauchy problem, observability, regularization, kernel-based learning theory.

\section{Introduction}
\label{sec:1}





The inverse Cauchy problem of the wave equation is a rich and active area of research from medical diagnostics to geophysical exploration, which involves determining the internal properties of a medium or structure by using information obtained from boundary measurements.
Over the past few decades, researchers have achieved significant developments in inverse problems of deterministic wave equations (see
\cite{BH,BY2017,KL,Uhlmann} and the rich refereces
therein), while stochastic disturbances are
common and non-negligible factors in real
applications. These stochasticities stem from
the complexity of the medium and the influence
of the external environment. In such situations,
stochastic wave equations are more accurately to
simulate the wave behavior. For instants, in the
model of the motion for a strand of DNA, elastic
forces are generally related to the second
derivative in the spatial variable, and the
molecular forces are reasonably modeled by a
stochastic noise term \cite{GM}; and in the
consideration of the internal structure of the
sun from the available observations of the
sun’s surface, waves going around the surface,
as well as shock waves propagating through the
sun itself, which cause the surface to pulsate
are caused by turbulence, since the location and
intensities of the shocks are unknown, a
probabilistic model can be considered
\cite{D1988}.

Compared to  inverse problems for deterministic hyperbolic equations, the randomness and uncertainty of both the model itself and the data bring more challenges to the study of  inverse problems of stochastic  hyperbolic equations.
The current research on inverse problems of stochastic  hyperbolic equations is still in its early stages. The main inherent difficulties caused by randomness and uncertainty make it difficult to directly extend traditional methods used to study problems of deterministic  hyperbolic equations to  stochastic ones.
For inverse problems related to stochastic  hyperbolic equations in the framework of random media with time-reversibility, we can refer to \cite{BR,BPTB,FGPS} and the extensive researches. In other cases, we face to deal with the time irreversibility cause by randomness.
Recently,  researches of inverse problems for stochastic  hyperbolic equations mainly focus on the inverse source problems and potential identifications of stochastic time-harmonic acoustic, electromagnetic, and elastic wave equations and made significant progress \cite{BCL,LHL,LL,LLW,LLM,LLar,LW}.
In these works, the properties of mild solutions are significantly used, and also, pseudo-differential operators, micro-local analysis, and regularization method, such as Karzmarz method,  have been applied to prove the uniqueness of the expectation and variance of random sources and random potential functions and constructed approximate solutions  from the boundary measurement of the radiated random wave field at multiple frequencies.
These methods also have been extended to the study of source identification in the stochastic wave equation by measurement data at a certain time \cite{FZLW}.
Gaussian beam solutions, which is a powerful tool in the study of uniqueness and stability of deterministic wave equation, have also adapted to identifying the scattering relation of stochastic wave equations on a Riemannian manifold through the emission of random sources \cite{HLO2014}.

In this paper, we study the problem of recovering the solution of stochastic hyperbolic equation from only boundary measurement of the wave field. The  inverse Cauchy problems have played in important role in the problem for reconstructing images of internal body structures in medical imaging, detecting flaws or inhomogeneities in materials and structures without causing damage, and exploring natural resources and understanding seismic activity in geophysics. Hence, the study for these problems have received ever increasing attention.

We first introduce some notations and spaces.

Let $T > 0$, $G \in \mathbb{R}^{n}$ ($n
\in \mathbb{N}$) be a given bounded domain
with the $C^{2}$ boundary $\G$. Put
\begin{equation}\label{6.1}
Q \= (0,T) \t G,\qq \Si \= (0,T) \t \G.
\end{equation}
For simplicity, we  use the notation $\ds z_i
\equiv z_{i}(x) \= \frac{\partial z(x)}{\partial
x_i}$, where $x_i$ is the $i$-th coordinate of a
generic point $x=(x_1,\cdots, x_n)$ in
$\mathbb{R}^{n}$. In a similar manner, we  use
notations $u_i$, $v_i$, etc. for the partial
derivatives of $u$ and $v$ with respect to
$x_i$. Also, we denote by $\nu(x) = (\nu^1(x),
\cdots, \nu^n(x))$ the unit outward normal
vector of $\G$ at point $x$. Denote by $T_x\G$ the tangent space of $\G$ at point $x$ and by $\{\mathbf{T}^1(x),\cds,\mathbf{T}^{n-1}(x)\}$  an orthonormal basis of $T_x\G$. Write $\pa_{\mathbf{T}^j} z$ the directive derivative of $z$ along $\tau^j$. Then, $\ds\nabla z = \sum_{j=1}^{n-1}\pa_{\mathbf{T}^j} z\mathbf{T}^j + \pa_\nu z \nu$. Next, denote $\mathbf{n}^j=(0,\cds,0,1,0,\cds,0)$, then $y_j = \nabla y\cd \mathbf{n}^j$.

Let $(\Om, {\cal F}, \dbF,
\dbP)$ with $\dbF=\{{\cal F}_t\}_{t \geq 0}$ be a complete filtered probability space on
which a  one dimensional standard Brownian
motion $\{W(t)\}_{t\geq 0}$ is defined, and $H$ be a Banach space. Denote by $L^{2}_{\dbF}(0,T;H)$ the Banach space consisting of all
$H$-valued and $\dbF$-adapted processes $X(\cdot)$ such that
$\mathbb{E}(|X(\cdot)|^2_{L^2(0,T;H)}) <
\infty$, by $L^{\infty}_{\dbF}(0,T;H)$ the
Banach space consisting of all $H$-valued and
$\dbF$-adapted bounded
processes, and by $L^{2}_{\dbF}(\Om;C([0,T];H))$ the Banach space consisting
of all $H$-valued and $\dbF$-adapted processes $X(\cdot)$ such that
$\mathbb{E}(|X(\cdot)|^2_{C(0,T;H)}) <
\infty$(similarly, one can define $L^{2}_{\dbF}(\Om;C^{k}([0,T];H))$ for any positive integer
$k$). All of these spaces are endowed with the
canonical norm.

Consider the following linear stochastic hyperbolic equation
\begin{eqnarray}\label{system1}
\left\{
\begin{array}{lll}\ds
\ds dz_{t} - \sum_{i,j = 1}^{n}(b^{ij}z_i)_{j}dt
= \big(b_1 z_t + b_2\cd\nabla z  + b_3 z + f
\big)dt +  b_4 z dW(t) & {\mbox { in }} Q,
 \\
\ns\ds  z = h_1,\q  & \mbox{ on } \Si,
\end{array}
\right.
\end{eqnarray}
where $\ds b_1 \in
L_{\dbF}^{\infty}(0,T;L^{\infty}(G))$,  $b_2 \in
L_{\dbF}^{\infty}(0,T;L^{\infty}(G;\mathbb{R}^{n}))$,
$b_3 \in L_{\dbF}^{\infty}(0,T;L^{p}(G))\,(p\in
[n,\infty])$, $b_4 \in
L_{\dbF}^{\infty}(0,T;L^{\infty}(G))$, $f \in
L^2_{\dbF}(0,T;L^2(G))$, $h_1\in
C^1_{\dbF}(0,T;L^2(\G))\cap
L^2_{\dbF}(0,T;H^1(\G))$ and $b^{ij} \in C^1(G)$
$(i,j = 1,2,\cdots, n)$ satisfy the following
assumptions:
\begin{condition}\label{condition of d}
\begin{enumerate}
\item[1.] $b^{ij} = b^{ji}$ $(i,j =
1,2,\cdots, n)$;

\item[2.]  For some constant $s_0
> 0$,
\begin{equation*}\label{bij}
 \sum_{i,j = 1}^{n}b^{ij}\xi^{i}\xi^{j} \geq s_0 |\xi|^2,
\,\,\,\,\,\,\,\,\,\, \forall\, (x,\xi)\=
(x,\xi^{1}, \cdots, \xi^{n}) \in G \t
\mathbb{R}^{n};
\end{equation*}

\item[3.]   There exists a positive function
$\f(\cdot) \in C^2(\overline{G})$ satisfying the
following:

\begin{enumerate}
\item[(1).] For some constant $\mu_0 >
0$, it holds
\begin{equation}\label{d1}
\begin{array}{ll}\ds
\sum_{i,j = 1}^{n}\Big\{ \sum_{i',j'}^n\Big[
2b^{ij'}(b^{i'j}\f_{i'})_{j'} -
b^{ij}_{j'}b^{i'j'}\f_{i'} \Big]
\Big\}\xi^{i}\xi^{j} \geq \mu_0
\sum_{i,j = 1}^{n}b^{ij}\xi^{i}\xi^{j}, \\
\ns\ds \hspace{5.5cm} \forall\,
(x,\xi^{1},\cdots,\xi^{n}) \in  \overline{G}  \t
\mathbb{R}^n.
\end{array}
\end{equation}

\item[(2).]  There is no critical point of
$\f(\cdot)$ in $\overline{G}$, i.e.,
\be\label{d2}\min_{x\in \overline{G} }|\nabla
\f(x)| > 0. \ee
\end{enumerate}
\end{enumerate}
\end{condition}

%

It is easy to check that if $\varphi(\cdot) $
satisfies Condition \ref{condition of d}, then
for any given constants $a \geq 1$ and $b \in
\mathbb{R}$, the function $\tilde{\varphi} = a\varphi + b$
still satisfies Condition  \ref{condition of d}
with $\mu_0$ replaced by $a\mu_0$. Therefore  we
may choose $\varphi$, $\mu_0$, $c_0>0$, $c_1>0$ and
$T$ such that  the following  condition
holds:
\begin{condition}\label{condition2}

\begin{enumerate} \label{con2 eq1}
\item[1.]
\begin{equation}\label{con2 eq1}
\frac{1}{4}\sum_{i,j = 1}^{n}b^{ij}(x)\f_i(x)\f_j(x) \geq
R^2_1\=\max_{x\in\overline G}\f(x)\geq R_0^2 \=
\min_{x\in\overline G}\f(x),\q \forall x\in
\overline G;
\end{equation}
\item[2.]  $T> T_0\=2 R_1;$

\item[3.] $\ds\(\frac{2R_1}{T}\)^2<c_1<\frac{2R_1}{T};$

\item[4.]  $\mu_0 - 4c_1 -c_0 > 0$.
\end{enumerate}
\end{condition}

\begin{remark}\label{rm2}
As we have explained, since
$\ds\sum_{i,j = 1}^{n}b^{ij}\f_i\f_j >0$, 
it is evident that by choosing $\mu_0$  sufficiently large, Condition \ref{condition2} can be satisfied. We include this to emphasize the relationship between $0<c_0< c_1 <1$, $\mu_0$ and $T$.
\end{remark}
%

%
%

\ms
The problem we considered in the paper can be conclude as follows:
\ms

\noindent\textbf{Problem (C)}.   Find the function $z$ solving \eqref{system1} with $\pa_\nu z=h_2$ on $\Si$ for a given $h_2\in L^2_{\dbF}(0,T;L^2(\G))$.

\ms

We first establish an observability estimate for the problem to show that the problem is conditional stability under suitable assumption of the solution. The main tool we used in this step is the Carleman estimate, which is a powerful tool in the study of second-order deterministic partial differential problems. We refer \cite{BY2017,Kli2013} for the application of Carleman estimates and  inverse problems for general hyperbolic systems with a single boundary measurement and \cite{Kli2015} for ill-posed Cauchy problems for deterministic partial differential equations.
At present, Carleman estimate has been introduced to stochastic partial differential equations, and thus provided opportunity to  study inverse problems and controllability  of stochastic partial differential equations \cite{Lu2021}.  In particular, the pointwise Carleman  identity was first established in \cite{Zhangxu3}. The inverse problem of the stochastic wave equation, including source or/and initial value identification problems with boundary observations, combining with observation data at a certain time or subregion of spatial domains  are studied in \cite{Lu1,LY2020,LZ2,Yuan}. From the perspective of practical applications, it is of great significance to study the inverse problems of stochastic hyperbolic equations with only boundary observation data.

To provide a convergent approximation to the problem, we apply Tikhonov regularization strategy, and construct a global concave functional, such that the minimizer of the functional uniquely exists. The existence and uniqueness, as well as the convergence rate of the minimizer are proved by the variational principle and Riesz representation theorem. We propose a numerical algorithm based on the proposed regularization and kernel-based learning theory \cite{Muller}, to avoid the deduction of adjoint problem, since it is much more difficult than the one for deterministic hyperbolic partial equation. Since Green's functions as the kernel are independent of meshes, thus there are more distinct advantages of the kernel-based approximation method, such as it can be easily numerically implement, especially for irregular domains.

This work is not a trivial extension to the previous research. While there are extensive research on the Cauchy problem for deterministic hyperbolic equations, the study of stochastic hyperbolic equations remains relatively unexplored due to their inherent randomness and uncertainty. Additionally, the properties of stochastic hyperbolic equations differ significantly from those of deterministic ones. For instance, the non-homogeneous terms in the diffusion terms of stochastic hyperbolic equations can be determined solely with boundary observations, whereas  source terms  in the deterministic hyperbolic equations cannot generally be identified through boundary observations alone \cite{LZ2}.  Moreover, unlike recent studies on inverse problems of stochastic hyperbolic equations, our work requires only boundary observations, without the need for initial or interior data.

The paper is organized as follows. In section 2, we devote to the observability estimate for the problem. Section 3 provides the proposed regularization based on the Tikhonov regularization method, as well as the properties of the regularized solution.  Numerical algorithms and experiments are presented in section 4 to illustrate the performance of the proposed method. At last,  we conclude the paper in section 5.

\ms

\section{Observability estimate}

Now we give a boundary observability estimate for the system \eqref{system1}.

\begin{theorem}\label{observability}
Let Conditions  \ref{condition of d} and
\ref{condition2} be satisfied. For any solution
of the equation
 \eqref{system1}, we have
\begin{equation} \label{obser esti2}
\begin{array}{ll}\ds
\q |(z,z_t)|_{L^2_\dbF(0,T;H^1(G))\times
L^2_\dbF(0,T;L^2(G))}
\\ \ns\ds  \leq Ce^{C(r_1^2 +1)}
\Big(|h_1|_{L^2_{\dbF}(0,T;H^1(\G))}+|h_{1,t}|_{L^2_{\dbF}(0,T;H^2(\G))}+|h_2|_{L^2_{\dbF}(0,T;L^2(\G))}
+ |f|_{L^2_{\dbF}(0,T;L^2(G))} \Big),
\end{array}
\end{equation}
where
\begin{eqnarray}\label{r1r2}
r_1 = |b_2|_{L_{\dbF}^{\infty}(0,T;L^{\infty}(G;\mathbb{R}^{n}))} +
|(b_1,b_3,b_4)|_{L_{\dbF}^{\infty}(0,T;L^{\infty}(G))^3}.
\end{eqnarray}
\end{theorem}
\begin{proposition}\label{hidden r}
For any solution of the equation \eqref{system1}, it
holds that
\begin{equation}\label{hidden ine}
\begin{array}{ll}\ds
| \partial_\nu z |_{L^2_{\dbF}(0,T;L^2(\G))}
\leq C\Big(|(z_t,z)|_{L^2_\dbF(0,T; L^2(G)\times
H^1(G))} +|h_1|_{L^2_{\dbF}(0,T;H^1(\G))}   +
|f|_{L^2_{\dbF}(0,T;L^2(G))} \Big).
\end{array}
\end{equation}
\end{proposition}

\begin{remark}
In \cite{Lu1}, the author proved
Proposition \ref{hidden r} in case that
$h_1=0$. The proof of Proposition \ref{hidden r}
for a general $h_1$ is
very similar. We only give a sketch of it here.
\end{remark}

\begin{proof} Since $\G\in C^2$, one can find a vector field
$\xi=(\xi^1,\cdots,\xi^n)\in
C^1(\mathbb{R}^n,\mathbb{R}^n)$ such that
$\xi=\nu$ on $\G$(see \cite[page 18]{Komornik}). By direct
computation, we can show that
\begin{equation}\label{equality hidden1}
\begin{array}{ll}
\ds \q \sum_{i=1}^n\Big[ 2(\xi\cd\nabla z)\sum_{j=1}^n b^{ij}z_{j} +
\xi^i\Big( z_t^2 - \sum_{j,k=1}^n
b^{jk}z_{j} z_{k} \Big) \Big]_{i}dt\\
\ns =  \ds -\Big\{2 \Big[\Big(dz_t - \sum_{i,j=1}^n
(b^{ij}z_{i})_{j}dt\Big) \xi \cd \nabla z -
d(z_t \xi\cd \nabla z)  -
\sum_{i,j,k=1}^n b^{ij}z_{i} z_{k}
\xi^k_{j}dt\Big] \\
\ns \ds\qq  - (\div \xi) z_t^2dt + \sum_{i,j=1}^n z_{j} z_{i}
\div(b^{ij}\xi)dt\Big\}.
\end{array}
\end{equation}

Let us analyze the terms in the left hand side of \eqref{equality hidden1}. Noting that $\nu\cd \mathbf{n}^j = \nu^j$ and $\xi=\nu$ on $\G$, we have that
\begin{equation}\label{equality hidden1-1}
\begin{array}{ll}
\ds \q  \mE\int_Q\sum_{i=1}^n\Big[ 2(\xi\cd\nabla z)\sum_{j=1}^n b^{ij}z_{j} \Big]_{i}dxdt \\
\ns\ds = 2 \mE\int_{\Si} \pa_\nu z\sum_{i,j=1}^n b^{ij} z_{j} \nu^i d\G dt \\
\ns\ds = 2 \mE\int_{\Si}\pa_\nu z\sum_{i,j=1}^n b^{ij} \nabla z\cd \mathbf{n}^j \nu^i d\G dt \\
\ns\ds = 2 \mE\int_{\Si} \pa_\nu z\sum_{i,j=1}^n b^{ij} \(\pa_\nu z \nu + \sum_{k=1}^{n-1}\pa_{\mathbf{T}^k}z\mathbf{T}^k \)\cd \mathbf{n}^j \nu^i d\G dt \\
\ns\ds = 2 \mE\int_{\Si}\(\sum_{i,j=1}^n b^{ij}  \nu^i\nu^j\) |\pa_\nu z|^2d\G dt + 2 \mE\int_{\Si} \pa_\nu z\sum_{i,j=1}^n b^{ij} \( \sum_{k=1}^{n-1}\pa_{\mathbf{T}^k}z\mathbf{T}^k \)\cd \mathbf{n}^j \nu^i d\G dt \\
\ns\ds \geq 2 \mE\int_{\Si}\(\sum_{i,j=1}^n b^{ij}  \nu^i\nu^j\) |\pa_\nu z|^2d\G dt - \e\mE\int_{\Si} |\pa_\nu z|^2d\G dt -  C|h_1|_{L^2_\dbF(0,T;H^1(\G))}^2,
\end{array}
\end{equation}
\begin{equation}\label{equality hidden1-2}
\ds \q  \mE\int_Q\sum_{i=1}^n\big(\xi^i z_t^2\big)_{i}dxdt  =  \mE\int_{\Si}\(\sum_{i=1}^n\xi^i\nu^i\)z_t^2d\G dt =  \mE\int_{\Si} z_t^2d\G dt,
\end{equation}
and
\begin{eqnarray}\label{equality hidden1-3}
&& \q
-\mE\int_Q\sum_{i=1}^n\Big(\xi^i\sum_{j,k=1}^n
b^{jk}z_{j} z_{k} \Big)_{i}dxdt \nonumber\\
&& =
-\mE\int_Q\sum_{k=1}^n\Big(\xi^k\sum_{i,j=1}^n
b^{ij}z_{i} z_{j} \Big)_{k}dxdt \nonumber\\
&& = -\mE\int_{\Si}\sum_{k=1}^n\xi^k\nu^k \!
\sum_{i,j=1}^n\!
b^{ij}\[\(\pa_\nu z \nu \!+\! \sum_{k=1}^{n-1}\pa_{\mathbf{T}^k}z\mathbf{T}^k \)\cd \mathbf{n}^i\]\! \[\(\pa_\nu z \nu \!+\! \sum_{k=1}^{n-1}\pa_{\mathbf{T}^k}z\mathbf{T}^k \)\cd \mathbf{n}^j \Big] d\G dt\nonumber\\
&& = -\mE\int_{\Si}\(\sum_{i,j=1}^n b^{ij}  \nu^i\nu^j\) |\pa_\nu z|^2d\G dt - 2\mE\int_{\Si} \sum_{i,j=1}^n\! b^{ij} \pa_\nu z \nu^i \( \sum_{k=1}^{n-1}\pa_{\mathbf{T}^k}z\mathbf{T}^k \)\cd \mathbf{n}^jd\G dt  \nonumber\\
&& \q -\mE\!\int_{\Si}\sum_{i,j=1}^n\!
b^{ij}\[\( \sum_{k=1}^{n-1}\pa_{\mathbf{T}^k}z\mathbf{T}^k \)\cd \mathbf{n}^i\]  \[\( \sum_{k=1}^{n-1}\pa_{\mathbf{T}^k}z\mathbf{T}^k \)\cd \mathbf{n}^j \Big] d\G dt\\
&& \geq - \mE\int_{\Si}  \(\sum_{i,j=1}^n b^{ij}
\nu^i\nu^j\) |\pa_\nu z|^2d\G dt -
\e\mE\int_{\Si}|\pa_\nu z|^2d\G dt
-C|h_1|_{L^2_\dbF(0,T;H^1(\G))}^2.\nonumber
\end{eqnarray}
Hence, by taking $\e=\frac{s_0}{4}$ we get that
\begin{equation}\label{equality hidden1-4}
\begin{array}{ll}
\ds \q  \mE\int_Q\sum_{i=1}^n\Big[ 2(\xi\cd\nabla z)\sum_{j=1}^n b^{ij}z_{j} +
\xi^i\Big( z_t^2 - \sum_{j,k=1}^n
b^{jk}z_{j} z_{k} \Big) \Big]_{i}dxdt \\
\ns\ds \geq  \mE\int_{\Si}\(\sum_{i,j=1}^n b^{ij}  \nu^i\nu^j\) |\pa_\nu z|^2d\G dt - 2\e\mE\int_{\Si} |\pa_\nu z|^2d\G dt -C|h_1|_{L^2_\dbF(0,T;H^1(\G))}^2 \\
\ns\ds \geq  \frac{s_0}{2}\mE\int_{\Si}  |\pa_\nu z|^2d\G dt  -C|h_1|_{L^2_\dbF(0,T;H^1(\G))}^2.
\end{array}
\end{equation}

At last, by integrating the right hand side of \eqref{equality hidden1}
in $Q$ and taking expectation in $\Om$, noting that $z$ solves \eqref{system1}, we get that
\begin{equation}\label{equality hidden1-1}
\begin{array}{ll}
\ds  \frac{s_0}{2}\mE\int_{\Si}|\pa_\nu z|^2d\G dt
-C|h_1|_{L^2_\dbF(0,T;H^1(\G))}^2 \leq  \ds C\mE\int_Q (|z_t|^2 +
|\nabla z|^2 + |z|^2)dxdt.
\end{array}
\end{equation}
This implies \eqref{hidden ine} immediately.\end{proof}

Further, we give an energy estimate for the equation
\eqref{system1}, which plays an important role in
the proof of the  observability estimate.

\medskip

\begin{proposition}\label{energy ensi}
For any $z$ solves the equation \eqref{system1}, 
there holds
\begin{equation}\label{en esti}
\begin{array}{ll}
\ds \cE_3(t) \3n&\ds \leq \ds  Ce^{C(r_1^2 +1)T}
\cE(s) + C (|h_1|^2_{L^2_\dbF(0,T;H^1(\G))}   +
|h_{1,t}|^2_{L^2_\dbF(0,T;L^2(\G))}\\ \ns&\ds \q
+ |h_{2}|^2_{L^2_\dbF(0,T;L^2(\G))} +
|f|^2_{L^2_\dbF(0,T;L^2(G))}),
 \end{array}
\end{equation}
for any $0\leq s,t\leq T$, where
$$
\cE_3(t) \= \mE\int_G\( |z_t(t,x)|^2 + |\nabla
z(t,x)|^2 +  |z(t,x)|^2\)
dx.
$$

\end{proposition}

\medskip

\begin{proof} Without loss of generality, we
assume that $t \leq s$.

By means of It\^{o}'s formula, we have
$$
d(z^2_t) = 2z_t dz_t + (dz_t)^2,
$$
which implies that
\begin{equation}\label{en eq1}
\begin{array}{ll}
\ds \q\mE\int_G  \big(|z_t(s,x)|^2 + |z(s,x)|^2\big) dx -  \mE\int_G \big(|z_t(t,x)|^2 + |z(t,x)|^2\big) dx
\\
\ns\ds = 2\mE\int_t^s\int_\G \sum_{i,j=1}^n b^{ij}(x)z_i(\tau,x)\nu^j z_{t}(\tau,x)d\G d\tau -2\mE\int_t^s\int_G \sum_{i,j=1}^nb^{ij}(x)z_i(\tau,x) z_{jt}(\tau,x)dxd\tau
\\
\ns \ds \q  + 2\mE\!\int_t^s\!\int_G z_t(\tau,x)\(b_1(\tau,x) z_t(\tau,x)\!+\!
b_2(\tau,x)\!\cd\!\nabla z(\tau,x)\!  + \!b_3(\tau,x) z(\tau,x) \!+\! f(\tau,x) \)dxd\tau \\
\ns \ds \q + \mE\int_t^s\int_G |b_4(\tau,x)
z(\tau,x) |^2 dxd\tau + 2\mE\int_t^s\int_G
z_t(\tau,x)z(\tau,x)dxd\tau.
\end{array}
\end{equation}
Therefore, we obtain that
\begin{equation}\label{en eq2}
\begin{array}{ll}
\ds \q\mE\int_G  \(|z_t(s,x)|^2 + \sum_{i,j}^n b^{ij}(x)z_i(s,x)z_j(s,x) + |z(s,x)|^2 \)dx \\
\ns \ds \q  -  \mE\int_G  \(|z_t(t,x)|^2 + \sum_{i,j}^n b^{ij}(x)z_i(t,x)z_j(t,x) + |z(t,x)|^2 \) dx\\
\ns \ds = 2\mE\int_t^s\int_\G \sum_{i,j=1}^n b^{ij}(x)z_i(\tau,x)\nu^j z_{t}(\tau,x)d\G d\tau\\
\ns\ds\q + \mE\int_t^s\int_G z_t(\tau,x)\Big[b_1(\tau,x)
z_t(\tau,x)+
 b_2(\tau,x)\cd\nabla z(\tau,x)  + b_3(\tau,x) z(\tau,x) + f (\tau,x)\Big]dxd\tau \\
\ns  \ds \q + 2\mE\int_t^s\int_G \big[b_4(\tau,x) z(\tau,x)
 + g(\tau,x) \big]^2 dxd\tau+ 2 \mE\int_t^s\int_G z_t(\tau,x)z(\tau,x)dxd\tau\\
\ns\ds \leq  C (|h_1|^2_{L^2(0,T;H^1(\G))}  + |h_{1,t}|^2_{L^2(0,T;L^2(\G))} + |h_{2}|^2_{L^2(0,T;L^2(\G))} ) \\
\ns\ds \q+  C(r_1^2 + 1)\mE \int_t^s \int_G
\big[z_t(\tau,x)^2 + |\nabla z(\tau,x)|^2 +
z(\tau,x)^2\big]dxd\tau   \\
\ns \ds \q +  2 \mE \int_t^s \int_G
z_t(\tau,x)z(\tau,x)dxd\tau + 2\mE\int_t^s\int_G
f(\tau,x)^2  dxd\tau.
\end{array}
\end{equation}

From \eqref{en eq2} and the
property of $b^{ij}$($i,j=1,\cds,n$)(see
\eqref{bij}), we find that
\begin{equation}\label{en eq3}
\begin{array}{ll}
\ds  \cE_3(t) \3n&\ds  \leq C \Big[\cE(s) + (r_1^2
+1 ) \int_t^s \cE(\tau)d\tau +
\mE\int_t^s\int_G f^2 dxdt\\
\ns&\ds\qq +   |h_1|^2_{L^2_\dbF(0,T;H^1(\G))} +
|h_{1,t}|^2_{L^2_\dbF(0,T;L^2(\G))} +
|h_{2}|^2_{L^2_\dbF(0,T;L^2(\G))} \Big].
\end{array}
\end{equation}
This, together with backward Gronwall's inequality,
implies  the inequality \eqref{en esti}
immediately.
\end{proof}

Next,  we introduce the following known result, which plays a key
role in getting Theorem \ref{observability}.

\begin{lemma}\cite[Theorem 4.1]{Zhangxu3}\label{hyperbolic1}
Let $p^{ij} \in C^{1}((0,T)\t \mathbb{R}^n)$ satisfy
\begin{equation}
p^{ij} = p^{ji}, \qq i,j = 1,2,\cdots,n,\nonumber
\end{equation}
$l,\Psi \in C^2((0,T)\t\mathbb{R}^n)$. Assume
that $u$ is an $H^2_{loc}(\mathbb{R}^n)$-valued
and $\mathbb{F}$-adapted process such
that $u_t$ is an $L^2(\mathbb{R}^n)$-valued
semimartingale. Set $\theta = e^l$ and $v=\theta
u$. Then, for a.e. $x\in \mathbb{R}^n$ and $\mathbb{P}$-a.s.
$\om \in \Om$,
\begin{eqnarray}\label{hyperbolic2}
&\,&\q\theta \Big( -2\ell_t v_t +
2\sum_{i,j = 1}^{n}p^{ij}\ell_i v_j + \Psi v \Big)
\Big[ du_t - \sum_{i,j = 1}^{n}(p^{ij}u_i)_j dt \Big] \nonumber\\
&\,& \,\,\,\q+\sum_{i,j = 1}^{n}\Big[
\sum_{i',j'=1}^n\big(
2p^{ij}p^{i'j'}\ell_{i'}v_iv_{j'} -
p^{ij}p^{i'j'}\ell_iv_{i'}v_{j'}
\big) - 2p^{ij}\ell_t v_i v_t + p^{ij}\ell_i v_t^2 \nonumber\\
&\,& \qq \qq \qq \,\,\,\,+ \Psi p^{ij}v_i v - \Big( A\ell_i +
\frac{\Psi_i}{2}\Big)p^{ij}v^2 \Big]_j dt \\
&\,& \,\,\,\q +d\Big[ \sum_{i,j = 1}^{n}p^{ij}\ell_t
v_i v_j - 2\sum_{i,j = 1}^{n}p^{ij}\ell_iv_jv_t + \ell_t
v_t^2 - \Psi v_t v + \Big( A\ell_t +
\frac{\Psi_t}{2}\Big)v^2 \Big] \nonumber \\
&\,& = \Big\{ \Big[ \ell_{tt} + \sum_{i,j = 1}^{n}(p^{ij}\ell_i)_{j} - \Psi
\Big]v_t^2 - 2\sum_{i,j = 1}^{n}[(p^{ij}\ell_j)_t +
p^{ij}\ell_{tj}]v_iv_t \nonumber \\
&\,& \,\,\,\,\,\,\,\,+\sum_{i,j = 1}^{n} \Big[ (p^{ij}\ell_t)_t +
\sum_{i',j'=1}^n\Big(2p^{ij'}(p^{i'j}\ell_{i'})_{j'}-(p^{ij}p^{i'j'}\ell_{i'})_{j'}\Big)
+ \Psi p^{ij} \Big]v_iv_j \nonumber \\
&\,&\,\,\,\,\,\,\,\,+ Bv^2 + \Big( -2\ell_tv_t +
2\sum_{i,j = 1}^{n}p^{ij}\ell_iv_j + \Psi v \Big)^2\Big\} dt +
\theta^2\ell_t(du_t)^2, \nonumber
\end{eqnarray}
where $(du_t)^2$ denotes the quadratic
variation process of $u_t$, %
$$\ds
A\=(l_t^2 - \ell_{tt}) - \sum_{i,j =
1}^{n}(p^{ij}\ell_i\ell_j -p_j^{ij}\ell_i -
p^{ij}\ell_{ij})-\Psi, $$ $$B\=A\Psi +
(Al_t)_t-\sum_{i,j = 1}^{n}(Ap^{ij}\ell_i)_j +
\frac{1}{2}\Big[ \Psi_{tt} - \sum_{i,j =
1}^{n}(p^{ij}\Psi_i)_j \Big]. $$
\end{lemma}

Now we turn to the proof of Theorem \ref{observability}. The proof is similar to the one for Theorem 1.1 in \cite{Lu1}, and the main difference is how to handle the boundary terms. Hence, we only give a sketch of the similar part and provide all details for the different part.

\begin{proof}[Proof of Theorem \ref{observability}] We divide the proof into four steps.

\vspace{0.2cm}

{\bf Step 1}.
 On one hand,
from Condition \ref{condition2}, we know that
there is an $\e_1\in (0,\frac12)$ such that
\begin{equation}\label{e1}
\ell(t,x)\leq \l\(\frac{R_1^2}{2} - \frac{cT^2}{8}\)
< 0,\q\forall\, (t,x)\in \[ \(0, \frac{T}{2}-\e_1
T\)\bigcup \(\frac{T}{2}-\e_1 T,T\)  \]\times G.
\end{equation}
On the other hand, since
$$
\ell\( \frac{T}{2},x\) = \f(x)\geq R_0^2, \qq \forall x\in G,
$$
we can find an $\e_0\in (0,\e_1)$ such that
\begin{equation}\label{e0}
\ell(t,x)\geq \frac{R_0^2}{2},\q\forall (t,x)\in  \(\frac{T}{2}-\e_0 T,\frac{T}{2}+\e_0 T\) \times G.
\end{equation}

Now we choose a $\chi\in C^\infty_0[0,T]$ satisfying
\begin{equation}\label{chi}
\chi=1 \mbox{ in } \(\frac{T}{2}-\e_1 T,\frac{T}{2}+\e_1 T\).
\end{equation}
 Let $y=\chi z$ for $z$ solving the equation \eqref{system1}, then we know that $y$ is a solution to the following equation:
\begin{eqnarray}{\label{system2}}
\left\{
\begin{array}{lll}\ds
\ds dy_{t} - \sum_{i,j = 1}^{n}(b^{ij}y_i)_{j}dt
= \Big[b_1 y_t + (b_2,\nabla y) + b_3 y + \chi f
+ \a \Big]dt +  b_4 y dW(t) & {\mbox { in }} Q,
 \\
\ns\ds  y = \chi h_1 & \mbox{ on } \Si, \\
\ns\ds  y(0) = y(T) = 0, \q y_{t}(0) = y_t(T) =
0 & \mbox{ in } G.
\end{array}
\right.
\end{eqnarray}
Here $\a = \chi_{tt}z + 2\chi_t z_t - b_1\chi_t
z$.

\vspace{0.2cm}

{\bf Step 2.} We apply Lemma \ref{hyperbolic1} to the solution of the
equation  \eqref{system2}. In the present case,
we choose
$$
p^{ij} = b^{ij}, \qq \Psi = \ell_{tt} +
\sum_{i,j = 1}^{n}(b^{ij}\ell_i)_{j} - \l c_0,
$$
and then
estimate the terms in \eqref{hyperbolic2} one by
one.

Similar to the proof of  \cite[Theorem 1.1]{Lu1}, we have that
\begin{equation}\label{coeffvt} \Big[\ell_{tt} +
\sum_{i,j = 1}^{n}(b^{ij}\ell_i)_{j} -\Psi\Big]v_{t}^2 = \l c_0 v_{t}^2,
\end{equation}
\begin{equation}\label{bcoeffvtvi}
\sum_{i,j = 1}^{n}\big[(b^{ij}\ell_j)_t +
b^{ij}\ell_{tj}\big]v_i v_t = 0,
\end{equation}
\begin{equation}\label{vivj}
\begin{array}{ll}\ds
\q\sum_{i,j = 1}^{n}\Big\{ (b^{ij}\ell_t)_t + \sum_{i',j'=1}^n \big[
2b^{ij'}(b^{i'j}\ell_{i'})_{j'} - (b^{ij}b^{i'j'}\ell_{i'})_{j'} \big] +
\Psi b^{ij} \Big\}v_i v_j \nonumber\\
\ns\ds \geq \l (\mu_0 -4c_1 -
c_0)\sum_{i,j = 1}^{n}b^{ij}v_i v_j,
\end{array}
\end{equation}
and
\begin{eqnarray}\label{6.20-eq1}
 B \3n&\geq& 2(4c_1+c_0)\l^3 (4R_1^2 - c_1^2T^2) + O(\l^2).
\end{eqnarray}
From \eqref{6.20-eq1},  we know that there exists a $\l_0 > 0$ such that for any $\l
\geq \l_0$,
\begin{equation}\label{B ine}
Bv^2 \geq 8c_1  (4R_1^2 - c_1^2T^2) \l^3v^2.
\end{equation}

Since
$$
v(0,x)=\theta(0,x)y(0,x) = 0
$$
and
$$
v_t(0,x)=\theta_t(0,x)y(0,x)+\th(0,x)
y_t(0,x)=0,
$$
we know that at time
$t=0$, it holds that
$$
\sum_{i,j = 1}^{n}b^{ij}\ell_t v_i v_j -
2\sum_{i,j = 1}^{n}b^{ij}\ell_iv_jv_t + \ell_t v_t^2
- \Psi v_t v + \Big( A\ell_t +
\frac{\Psi_t}{2}\Big)v^2 =0.
$$
Similarly,  at time $t=T$, it holds that
$$
\sum_{i,j = 1}^{n}b^{ij}\ell_t v_i v_j -
2\sum_{i,j = 1}^{n}b^{ij}\ell_iv_jv_t + \ell_t v_t^2
- \Psi v_t v + \Big( A\ell_t +
\frac{\Psi_t}{2}\Big)v^2 = 0.
$$

{\bf Step 3.} In this step, we handle the boundary terms. Noting that $$v_i=\(\pa_\nu v \nu + \sum_{k=1}^{n-1}\pa_{\mathbf{T}^k}v\mathbf{T}^k \)\cd \mathbf{n}^i,$$
we have that
\begin{eqnarray}\label{bhyperbolic32}
&\,&\q \Big|\mathbb{E}\int_{Q}\sum_{i,j = 1}^{n}\Big(
\sum_{i',j'=1}^n
2b^{ij}b^{i'j'}\ell_{i'}v_iv_{j'} \)_jdxdt\Big|\nonumber\\
&\,& =\Big|\mathbb{E}\int_{\Si}\sum_{i,j = 1}^{n}\sum_{i',j'=1}^n\Big(
2b^{ij}b^{i'j'}\f_{i'}v_i v_{j'}
\Big)\nu^j d\Si \Big|  \\
&\,&\ds = \Big|\mathbb{E}\int_{\Si}\sum_{i,j = 1}^{n}
\sum_{i',j'=1}^n\Big[
2b^{ij}b^{i'j'}\f_{i'}\(\pa_\nu v \nu + \sum_{k=1}^{n-1}\pa_{\mathbf{T}^k}v\mathbf{T}^k \)\cd \mathbf{n}^i\(\pa_\nu v \nu + \sum_{k=1}^{n-1}\pa_{\mathbf{T}^k}v\mathbf{T}^k \)\cd \mathbf{n}^{j'} \]
 \nu^j d\Si\Big|  \nonumber\\
&\,&\ds \leq C\l\(\mathbb{E}\int_{\Si}|\pa_\nu v|^2 d\Si + \mathbb{E}\int_{\Si}\Big|\sum_{k=1}^{n-1}\pa_{\mathbf{T}^k}v\mathbf{T}^k\Big|^2 d\Si\).\nonumber
\end{eqnarray}
Similarly,
\begin{equation}\label{6.20-eq2}
\begin{array}{ll}\ds
\Big|\mE\int_Q\sum_{i,j = 1}^{n}
\big(b^{ij}b^{i'j'}\ell_iv_{i'}v_{j'}
 \big)_j dxdt\Big|\\
\ns\ds  =\Big|\mE\int_\Si\sum_{i,j = 1}^{n}
\sum_{i',j'=1}^n\big(
b^{ij}b^{i'j'}\ell_iv_{i'}v_{j'}
\big) \nu^j d\Si\Big|\\
\ns\ds \leq C\l\(\mathbb{E}\int_{\Si}|\pa_\nu v|^2 d\Si + \mathbb{E}\int_{\Si}\Big|\sum_{k=1}^{n-1}\pa_{\mathbf{T}^k}v\mathbf{T}^k\Big|^2 d\Si\).
\end{array}
\end{equation}
Next,
\begin{equation}\label{6.20-eq3}
\begin{array}{ll}\ds
\Big|\mE\int_Q\sum_{i,j = 1}^{n}
 \big( -2b^{ij}\ell_t v_i v_t + b^{ij}\ell_i v_t^2  \big)_j dxdt\Big|\\
\ns\ds  =\Big|\mE\int_\Si\sum_{i,j = 1}^{n}
 \big( -2b^{ij}\ell_t v_i v_t + b^{ij}\ell_i v_t^2  \big)\nu^j d\Si\Big|\\
\ns\ds \leq \Big|\mE\int_\Si\sum_{i,j = 1}^{n}
 \[ -2b^{ij}\ell_t \(\pa_\nu v \nu + \sum_{k=1}^{n-1}\pa_{\mathbf{T}^k}v\mathbf{T}^k \)\cd \mathbf{n}^i v_t + b^{ij}\ell_i v_t^2  \]\nu^j d\Si\Big|\\
\ns\ds \leq C\l\(\mathbb{E}\int_{\Si}|\pa_\nu v|^2 d\Si + \mathbb{E}\int_{\Si}\Big|\sum_{k=1}^{n-1}\pa_{\mathbf{T}^k}v\mathbf{T}^k\Big|^2 d\Si + \mathbb{E}\int_{\Si}|v_t|^2 d\Si\)
\end{array}
\end{equation}
and
\begin{equation}\label{6.20-eq4}
\begin{array}{ll}\ds
\Big|\mE\int_Q\sum_{i,j = 1}^{n}\Big[
 \Psi b^{ij}v_i v - \Big( A\ell_i +
\frac{\Psi_i}{2}\Big)b^{ij}v^2 \Big]_j dxdt\Big|\\
\ns\ds  =\Big|\mE\int_\Si\sum_{i,j = 1}^{n}\Big[
 \Psi b^{ij}v_i v - \Big( A\ell_i +
\frac{\Psi_i}{2}\Big)b^{ij}v^2 \Big]\nu^j d\Si\Big|\\
\ns\ds \leq C\l\(\mathbb{E}\int_{\Si}|\pa_\nu v|^2 d\Si + \mathbb{E}\int_{\Si}\Big|\sum_{k=1}^{n-1}\pa_{\mathbf{T}^k}v\mathbf{T}^k\Big|^2 d\Si + \l^2\mathbb{E}\int_{\Si}|v|^2 d\Si\).
\end{array}
\end{equation}

From \eqref{bhyperbolic32} to \eqref{6.20-eq4}, we find that
\begin{equation}\label{6.20-eq5}
\begin{array}{ll}\ds
\Big|\mE\int_Q\sum_{i,j = 1}^{n}\Big[
\sum_{i',j'=1}^n\big(
2b^{ij}b^{i'j'}\ell_{i'}v_iv_{j'} -
b^{ij}b^{i'j'}\ell_iv_{i'}v_{j'}
\big) - 2b^{ij}\ell_t v_i v_t + b^{ij}\ell_i v_t^2 \\
\ns\ds \qq \qq \qq \,\,\,\,+ \Psi b^{ij}v_i v - \Big( A\ell_i +
\frac{\Psi_i}{2}\Big)b^{ij}v^2 \Big]_j dxdt\Big|\\
\ns\ds \leq C\l\(\mathbb{E}\int_{\Si}|\pa_\nu v|^2 d\Si + \mathbb{E}\int_{\Si}\Big|\sum_{k=1}^{n-1}\pa_{\mathbf{T}^k}v\mathbf{T}^k\Big|^2 d\Si+ \mathbb{E}\int_{\Si}|v_t|^2 d\Si + \l^2\mathbb{E}\int_{\Si}|v|^2 d\Si\).
\end{array}
\end{equation}
Noting that $\pa_\nu v = (\pa_\nu \th)y + \th \pa_\nu y$, we have that
\begin{equation}\label{6.20-eq6}
\begin{array}{ll}\ds
\mathbb{E}\int_{\Si}|\pa_\nu v|^2 d\Si\3n&\ds = \mathbb{E}\int_{\Si}|(\pa_\nu \th)y + \th \pa_\nu y|^2 d\Si\\
\ns&\ds \leq C\(\mathbb{E}\int_{\Si}\th^2|\pa_\nu y|^2 d\Si +  \l^2\mathbb{E}\int_{\Si}\th^2| y|^2 d\Si\).
\end{array}
\end{equation}
Similarly, from  $\pa_{\mathbf{T}^k}v  = (\pa_{\mathbf{T}^k} \th)y + \th \pa_{\mathbf{T}^k} y$, we obtain that
\begin{equation}\label{6.20-eq7}
\begin{array}{ll}\ds
\mathbb{E}\int_{\Si}\Big|\sum_{k=1}^{n-1}\pa_{\mathbf{T}^k}v\mathbf{T}^k\Big|^2 d\Si\3n&\ds = \mathbb{E}\int_{\Si}\Big|\sum_{k=1}^{n-1}\big[(\pa_{\mathbf{T}^k} \th)y + \th \pa_{\mathbf{T}^k} y\big]\mathbf{T}^k\Big|^2 d\Si\\
\ns&\ds \leq C\(\mathbb{E}\int_{\Si}\th^2\Big|\sum_{k=1}^{n-1}\pa_{\mathbf{T}^k}y\mathbf{T}^k\Big|^2 d\Si +  \l^2\mathbb{E}\int_{\Si}\th^2| y|^2 d\Si\).
\end{array}
\end{equation}
Next,
\begin{equation}\label{6.20-eq8}
\begin{array}{ll}\ds
\mathbb{E}\int_{\Si}|v_t|^2 d\Si = \mathbb{E}\int_{\Si}| \th_ty + \th  y_t|^2 d\Si \leq C\(\mathbb{E}\int_{\Si}\th^2|y_t|^2 d\Si +  \l^2\mathbb{E}\int_{\Si}\th^2| y|^2 d\Si\).
\end{array}
\end{equation}

Combining \eqref{6.20-eq6}--\eqref{6.20-eq8}, we find that
\begin{equation}\label{6.20-eq9}
\begin{array}{ll}\ds
 \(\mathbb{E}\int_{\Si}|\pa_\nu v|^2 d\Si + \mathbb{E}\int_{\Si}\Big|\sum_{k=1}^{n-1}\pa_{\mathbf{T}^k}v\mathbf{T}^k\Big|^2 d\Si+ \mathbb{E}\int_{\Si}|v_t|^2 d\Si + \l^2\mathbb{E}\int_{\Si}|v|^2 d\Si\)\\
\ns\ds \leq C \[\mathbb{E}\int_{\Si}\th^2\(|\pa_\nu y|^2+\Big|\sum_{k=1}^{n-1}\pa_{\mathbf{T}^k}y\mathbf{T}^k\Big|^2  +  |y_t|^2  + \l^2 |y|^2\) d\Si\].
\end{array}
\end{equation}

{\bf Step 4.} Integrating
(\ref{hyperbolic2}) in $Q$, taking
expectation in $\Om$ and by the argument
above, we know that there exists two constants $C_1, C_2>0$ such that
\begin{equation}\label{bhyperbolic3}
\begin{array}{ll}\ds
 \q\mathbb{E}\int_Q \theta\Big\{\Big( -2\ell_t v_t +
2\sum_{i,j = 1}^{n}b^{ij}\ell_iv_j + \Psi v \Big)
\Big[ dy_t - \sum_{i,j = 1}^{n}(b^{ij}y_i)_jdt \Big]
- \theta \ell_t (dy_t)^2\Big\}dx  \\
\ns\ds\,\,\,\,\,\,\, + C_1\l \mathbb{E}
\int_{\Si}\th^2\(|\pa_\nu
y|^2+\Big|\sum_{k=1}^{n-1}\pa_{\mathbf{T}^k}y
\mathbf{T}^k\Big|^2 +  |y_t|^2  + \l^2 |y|^2\)
d\Si
\\
\ns\ds \geq C_2 \mathbb{E}\int_Q
\theta^2\Big[\big( \l  v_t^2 + \l
|\nabla v|^2  \big) +   \l^3 v^2 \]dxdt
+ \mathbb{E}\int_Q\Big(-2\ell_tv_t +
2\sum_{i,j = 1}^{n}b^{ij}\ell_iv_j + \Psi v\Big)^2
dxdt.
\end{array}
\end{equation}
Since  $y$ solves the equation \eqref{system2},  we know that
\begin{eqnarray}\label{bhyperbolic4}
&\,&\q\mathbb{E}\int_Q \theta\Big\{\Big(
-2l_t v_t + 2\sum_{i,j = 1}^{n}b^{ij}\ell_iv_j +
\Psi v \Big) \Big[ dy_t -
\sum_{i,j = 1}^{n}(b^{ij}y_i)_jdt \Big] - \theta \ell_t (dy_t)^2\Big\}dx   \nonumber\\
&\,&\ds= \mathbb{E}\int_Q
\theta\Big\{\Big( -2\ell_t v_t +
2\sum_{i,j = 1}^{n}b^{ij}\ell_iv_j + \Psi v \Big)
\big[ b_1 y_t +  b_2\cd\nabla y + b_3 y
+ \chi  f + \a\big]
\\ &\,&\ds\qq\qq - \theta \ell_t |b_4 y|^2\Big\}dxdt
\nonumber
\\
&\,&\ds\leq  C\Big[ \mathbb{E}\int_Q \theta^2
\( b_1y_t + b_2\cd\nabla y + b_3 y + \chi f +
\a \)^2 + \l \theta^2|b_4 y|^2 \Big]dxdt
\nonumber
\\
&\,&\ds \qq +\mathbb{E}\int_Q\(-2\ell_t v_t
+ \sum_{i,j = 1}^{n}b^{ij}\ell_i v_j +
\Psi v \)^2dxdt  \nonumber\\
&\,&\ds \leq C\bigg\{\mathbb{E}\int_Q
\theta^2(f^2 + \a^2 )dxdt +
|b_1|^2_{L^{\infty}_{\dbF}(0,T;L^{\infty}(G))}
\mathbb{E}\int_Q \theta^2 y_t^2 dxdt  \nonumber
\\
&\,&\ds \qq\;\, +
|b_2|^2_{L^{\infty}_{\dbF}(0,T;L^{\infty}(G,\mathbb{R}^n))}  \mathbb{E}\int_Q\theta^2
|\nabla y|^2dxdt +  \l
|(b_3,b_4)|_{L^{\infty}_{\dbF}(0,T;L^{\infty}(G))^2}^2\mathbb{E}\int_Q
\theta^2 y^2 dxdt   \bigg\}\nonumber
 \\
&\,&\ds \q +\mathbb{E}\int_Q\(-2\ell_t v_t +
\sum_{i,j = 1}^{n}b^{ij}\ell_i v_j + \Psi v
\)^2dxdt.\nonumber
\end{eqnarray}

Taking $\l_2 = C(r_1^2 +1)\geq\max\{\l_0,\l_1\}$,
combining \eqref{bhyperbolic3} and
\eqref{bhyperbolic4},
for any $\l \geq \l_2$, we have that
\begin{equation}\label{bhyperbolic5}
\begin{array}{ll}
\ds\q  C \l \mathbb{E}\int_{\Si}\th^2\(|\pa_\nu y|^2+\Big|\sum_{k=1}^{n-1}\pa_{\mathbf{T}^k}y\mathbf{T}^k\Big|^2  +  |y_t|^2  + \l^2 |y|^2\) d\Si
+ C\mathbb{E}\int_Q \theta^2(f^2 + \a^2 )dxdt \\
\ns\ds \geq \mathbb{E}\int_Q \theta^2 \(  \l  y_t^2 + \l
|\nabla y|^2   +   \l^3 y^2  \)dxdt.
\end{array}
\end{equation}
Recalling the property of $\chi$(see \eqref{chi}) and $y=\chi z$, from \eqref{bhyperbolic5}, we find
\begin{equation}\label{bhyperbolic6}
\begin{array}{ll}
\ds\q  C \l \mathbb{E}\int_{\Si}\th^2\(|\pa_\nu y|^2+\Big|\sum_{k=1}^{n-1}\pa_{\mathbf{T}^k}y\mathbf{T}^k\Big|^2  +  |y_t|^2  + \l^2 |y|^2\) d\Si \\
\ns\ds \q  + C\mathbb{E}\int_Q \theta^2 f^2 dxdt + C(r_1 + 1)\[\mathbb{E}\int_{0}^{\frac{T}{2}-\e_1T} \int_G \theta^2 (z_t^2  + |\nabla z|^2  +  z^2)dxdt\\
\ns\ds \qq  +
\mathbb{E}\int_{\frac{T}{2}+\e_1T}^{ T}\int_G
\theta^2  \big(z_t^2 + |\nabla z|^2 +
z^2\big)dxdt\]
\\
\ns\ds \geq \mathbb{E}\int_{\frac{T}{2}-\e_0
T}^{\frac{T}{2}+\e_0 T}\int_G \theta^2 \big(  \l
z_t^2 + \l |\nabla z|^2   +   \l^3  z^2
\big)dxdt.
\end{array}
\end{equation}
From \eqref{bhyperbolic6} and \eqref{en esti}, we know that there is a $\l_3=C(r_1^2 +1)\geq \l_2$ such that for all $\l\geq \l_3$, it holds that
\begin{equation}\label{bhyperbolic7}
\begin{array}{ll}
\ds\q  C \l\mathbb{E}\int_{\Si}
\th^2\(|\pa_\nu y|^2+\Big|\sum_{k=1}^{n-1}\pa_{\mathbf{T}^k}y\mathbf{T}^k\Big|^2  +  |y_t|^2  + \l^2 |y|^2\) d\Si+ C\mathbb{E}\int_Q \theta^2 f^2 dxdt \\
\ns\ds \geq
e^{-\frac{T^2}{4}\l_3}\mathbb{E}\int_Q \theta^2
\big( z_t^2 + |\nabla z|^2   +    z^2
\big)dxdt.
\end{array}
\end{equation}
Taking $\l=\l_3$, we obtain that
\begin{equation}\label{bhyperbolic8}
\begin{array}{ll}
\ds\q  C e^{\l_3
R_1^2}\big(|h_1|^2_{L^2_\dbF(0,T;H^1(\G))} +
|h_{1,t}|^2_{L^2_\dbF(0,T;L^2(\G))} +
|h_{2}|^2_{L^2_\dbF(0,T;L^2(\G))} )+
|f|^2_{L^2_\dbF(0,T;L^2(G))}
\big) \\
\ns\ds \geq e^{-\frac{T^2}{4}\l_3}\mathbb{E}
\int_G  \big( z_1^2 + |\nabla z_0|^2
\big)dxdt.
\end{array}
\end{equation}
This leads to the inequality \eqref{obser esti2} immediately. \end{proof}

\section{Regularization}

\label{sec:6.3}

Recall that we need to find an approximate
solution  $u$ of \textbf{Problem (C)}.

Put
$$
\cH_T\deq \{u\in
L^{2}_\dbF\left(0,T;H^2(G)\right)\cap
H^{1}_\dbF\left(0,T;H^1(G)\right)\}.
$$
Clearly, $\cH_T$ is a Hilbert space with the
canonical norm.

Define an operator $L$ on
$L^{2}_\dbF\left(0,T;H^2(G)\right)\cap
H^{1}_\dbF\left(0,T;H^1(G)\right)$ as
\begin{equation}\label{6.20-eq10}
\begin{array}{ll}\ds
L\rho(t)\3n&\ds \= \rho_t(t) - \rho_t(0) -
\int_0^t\Big[\sum_{i,j =
1}^{n}(b^{ij}z_i)_{j}  + \big(b_1 z_t +
 b_2\cd\nabla z  + b_3 z  \big)\Big]dt\\
\ns&\ds\q - \int_0^t   b_4 z dW(t), \qq
\forall \rho\in
L^{2}_\dbF\left(0,T;H^2(G)\right)\cap
H^{1}_\dbF\left(0,T;H^1(G)\right).
\end{array}
\end{equation}

Put
\begin{equation}\label{6.27}
\cA_{ad}\deq\{\rho\in \cH_T:\, L\rho\in
L^2_\dbF(\Om;H^1(0,T;L^2(G))),\; \rho\mid
_{\Si}=h_{1},\;
\partial _{\nu}\rho\mid _{\Si}=h_{2}\}.
\end{equation}
Clearly, $\cA_{ad}$ is a Hilbert space with the
norm
$$
|\rho|_{\cA_{ad}}\deq
\big(|\rho|^2_{L^{2}_\dbF\left(0,T;H^2(G)\right)}
+ |\rho|^2_{H^{1}_\dbF\left(0,T;H^1(G)\right)} +
|L\rho|^2_{L^2_\dbF(\Om;H^1(0,T;L^2(G)))}
\big)^{\frac{1}{2}}.
$$

Let
$$
\mathbf{I}f=\int_0^t f(s)ds.
$$

Let us find an approximate solution of
\textbf{Problem (C)} via the minimization of the
following Tikhonov functional
\begin{eqnarray}\label{6.28}
\cJ_{\gamma }\left( u\right) = |L u-\mathbf{I}f
|_{L^2_\dbF(\Om;H^1(0,T;L^2(G)))} ^{2}+\gamma
|u|_{\cH_T}^{2},\qq u\in \cA_{ad}.
\end{eqnarray}

\begin{theorem}[Existence and uniqueness of the minimizer]\label{th2}
For every $\gamma \in \left( 0,1\right)$ there
exists a unique minimizer $u_{\gamma }\in
\cA_{ad}$ of the functional \eqref{6.28} and the
following estimate holds
\begin{equation}\label{4.12}
|u_{\gamma }|_{\cA_{ad}} \leq \frac{C
}{\sqrt{\gamma }} |f|_{L^{2}_\dbF(0,T;L^2(G))}.
\end{equation}
\end{theorem}

\begin{proof} Let $u_{\gamma }$ be a minimizer of the functional
\eqref{6.28}. Then for any $\rho \in \cA_{ad}$,
$$
\lim_{\e\to 0}\frac{\cJ_{\gamma }(u_\g +
\e\rho)-\cJ_{\gamma }( u_\g)}{\e}=0,
$$
which implies that
\begin{equation}\label{2.109}
\begin{array}{ll}\ds
\lan Lu_{\gamma
},L\rho\ran_{L^2_\dbF(\Om;H^1(0,T; L^2(G)))}
+\gamma \lan u_{\gamma },\rho\ran_{\cH_T} =\lan
L\rho, \mathbf{I}f  \ran_{L^2_\dbF(\Om;H^1(0,T;
L^2(G)))},\q \forall \rho\in \cA_{ad}.
\end{array}
\end{equation}
Let
\begin{equation}\label{2.110}
\begin{array}{ll}\ds
\lan u,\rho\ran_{\gamma }\deq\lan
Lu,L\rho\ran_{L^2_\dbF(\Om;H^1(0,T; L^2(G)))}
+\gamma \lan  u,\rho\ran_{\cH_T},  \qq \forall
u,\rho\in \cA_{ad}.
\end{array}
\end{equation}
Then, $\lan u,\rho\ran_{\gamma }$ defines a new
scalar product in the Hilbert space $\cA_{ad}$
and the corresponding norm $|u|_\g$ satisfies
\begin{equation}\label{2.111}
\sqrt{\gamma }|u|_{\cA_{ad}}\leq |u|_{\gamma
}\leq C|u|_{\cA_{ad}}.
\end{equation}
Thus, the  norm $|u|_{\gamma}$ is equivalent to
the norm $|u|_{\cA_{ad}}$. Hence, \eqref{2.109}
can be rewritten as
\begin{equation}\label{2.112}
\begin{array}{ll}\ds
\lan u_{\gamma },\rho\ran_{\gamma}=\lan L\rho,
\mathbf{I}f  \ran_{L^2_\dbF(\Om;H^1(0,T;
L^2(G)))},\q \forall \rho\in \cA_{ad}.
\end{array}
\end{equation}
It follows from (\ref{2.111}) that
\begin{equation}\label{2.113}
\begin{array}{ll}\ds
\big|\lan L\rho, \mathbf{I}f
\ran_{L^2_\dbF(\Om;H^1(0,T; L^2(G)))}\big|  \leq
C |f|_{L_{\dbF}^2(0,T;L^2(G))} |\rho|_{\gamma },
\end{array}
\end{equation}
which yields the right hand side of
(\ref{2.112}) is a bounded linear functional on
$\cA_{ad}$. By Riesz representation theorem,
there exists an element $w_{\gamma }\in
\cA_{ad}$ such that
$$
\begin{array}{ll}\ds
\lan L\rho, \mathbf{I}f
\ran_{L^2_\dbF(\Om;H^1(0,T; L^2(G)))} =\lan
w_{\gamma },\rho\ran_{\gamma }, \qq\forall
\rho\in \cA_{ad}.
\end{array}
$$
This, together  with \eqref{2.112},  implies
that
$$
\lan u_{\gamma },\rho\ran_{\gamma }=\lan
w_{\gamma },\rho\ran_{\gamma },\qq\forall
\rho\in \cA_{ad}.
$$
Hence, the minimizer  $ u_{\gamma }=w_{\gamma
}.$ Also, by Riesz  representation theorem and
(\ref{2.113}), we find that
$$
|u_{\gamma }| _{\gamma }\leq C
|f|_{L_{\dbF}^2(0,T;L^2(G))}.
$$
Hence, the minimizer $u_{\gamma }$ is unique and
the left inequality in \eqref{2.111} implies
\eqref{4.12}.
\end{proof}

To estimate the convergence rate of minimizers,
we assume that there exists the exact solution
$u^{\ast }\in \cA_{ad}$ of \textbf{Problem (C)}
with the exact right hand side $f^{\ast }$ in
\eqref{system1} and exact lateral Cauchy data
$h_{1}^{\ast }$ and $h_{2}^{\ast }$.

\begin{theorem}[Convergence rate]\label{th3}
For  $\gamma \in \left( 0,1\right) $, let
$u_{\gamma }\in \cH_T$  be the unique minimizer
of the functional (\ref{6.28}), which is
guaranteed by Theorem \ref{th2}. Suppose that
Conditions  \ref{condition of d} and
\ref{condition2} hold.  Then  the
following estimate holds
\begin{equation}\label{6.31}
\begin{array}{ll}\ds
|u_{\gamma }-u^{\ast
}|_{H^1_\dbF(0,T;L^2(G))\cap
L^2_\dbF(0,T;H^1(G))}
\\
\ns\ds \leq C \big(|h_{1}-h_{1}^{\ast
}|_{L^2_\dbF(0,T;H^1(\G))
}+|h_{1,t}-h_{1,t}^{\ast
}|_{L^2_\dbF(0,T;L^2(\G)) }+|h_{2}-h_{2}^{\ast
}|
_{L^2_\dbF(0,T;L^2(\G)) }\\
\ns\ds\qq +| f-f^{\ast }|_{L^2_\dbF(0,T;L^2(G))}
+\sqrt{\gamma }|u^{\ast }|_{\cH_T} \big).
\end{array}
\end{equation}
%
\end{theorem}
\begin{corollary}
Let $\delta \in
\left( 0,1\right) $, $\gamma =\gamma
\left( \delta
\right) =\delta ^{2}$ and let
\begin{equation*}
\begin{array}{ll}\ds
|h_{1}-h_{1}^{\ast }|_{L^2_\dbF(0,T;H^1(\G))
}+|h_{1,t}-h_{1,t}^{\ast
}|_{L^2_\dbF(0,T;L^2(\G)) }\\
\ns\ds +|h_{2}-h_{2}^{\ast }|
_{L^2_\dbF(0,T;L^2(\G)) }  +| f-f^{\ast
}|_{L^2(0,T;L^2(G))} \leq \delta.
\end{array}
\end{equation*}
Then we have
\begin{equation}\label{6.32}
|u_{\gamma }-u^{\ast }|
_{H^1_\dbF(0,T;L^2(G))\cap L^2_\dbF(0,T;H^1(G))}
\leq C\left( 1+|u^{\ast
}|_{\cA_{ad}}\right) \delta.
\end{equation}
\end{corollary}

\begin{proof}[Proof of Theorem \ref{th3}]
From the definition of $u^*$ and $f^*$, we have
that
\begin{equation}\label{2.115}
\begin{array}{ll}\ds
\lan Lu^{\ast },L\rho\ran_{L^2_\dbF(\Om;H^1(0,T;
L^2(G)))} +\gamma \lan u^{\ast
},\rho\ran_{\cH_T} \\
\ns\ds =\lan \mathbf{I}f^{\ast}, L\rho
\ran_{L^2_\dbF(\Om;H^1(0,T; L^2(G)))} +\gamma
\lan u^{\ast },\rho\ran_{\cH_T}, \qq\forall
\rho\in \cA_{ad}.
\end{array}
\end{equation}

Put
$$
\widetilde{u}_{\gamma }=u^{\ast } -u_{\gamma
},\q\widetilde{f}=f^{\ast }-f.
$$

Subtracting identity (\ref{2.109}) from identity
(\ref{2.115}) and denoting $
\widetilde{u}_{\gamma }=u^{\ast }-u_{\gamma }$,
we obtain
\begin{equation*}
\begin{array}{ll}\ds
\lan L\widetilde{u}_{\gamma },
L\rho\ran_{L^2_\dbF(\Om;H^1(0,T; L^2(G)))}
+\gamma \lan \widetilde{u}
_{\gamma }, \rho\ran_{\cH_T}\\
\ns\ds =\lan L\rho, \mathbf{I}\tilde f
\ran_{L^2_\dbF(\Om;H^1(0,T; L^2(G)))} +\gamma
\lan u^{\ast },\rho\ran_{\cH_T}, \qq\forall
\rho\in \cA_{ad}.
\end{array}
\end{equation*}
Setting here $\rho =\widetilde{u}_{\gamma }$, we
get
\begin{equation}\label{2.116}
\begin{array}{ll}
\ds |L\widetilde{u}_{\gamma
}|_{L^2_\dbF(\Om;H^1(0,T; L^2(G)))}^{2}+\gamma
|\widetilde{u}_{\gamma }|
_{\cH_T}^{2}\\
\ns\ds =\lan L\widetilde{u}_{\gamma },
\mathbf{I}\tilde f  \ran_{L^2_\dbF(\Om;H^1(0,T;
L^2(G)))} +\gamma \lan u^{\ast },\widetilde{u}
_{\gamma }\ran_{\cH_T}\\
\ns\ds\leq \frac{1}{2}| L\widetilde{u}_{\gamma
}|_{L^2_\dbF(\Om;H^1(0,T; L^2(G)))}^{2}
+\frac{1}{2}\big|\mathbf{I}\tilde f
\big|_{L^2_\dbF(\Om;H^1(0,T;
L^2(G)))}^{2}+\frac{\gamma }{2} |u^{\ast
}|_{\cH_T}^{2}+\frac{ \gamma }{2}|
\widetilde{u}_{\gamma }|_{\cH_T }^{2}.
\end{array}
\end{equation}
Collecting terms on both side of
\eqref{2.116}, we find that
\begin{equation}\label{6.34}
\begin{array}{ll}
|L \widetilde{u}_{\gamma
}|_{L^2_\dbF(\Om;H^1(0,T; L^2(G)))}^{2}+ \gamma
| \widetilde{v}_{\gamma }|_{\cH_T
}^{2}\3n&\ds\leq \big|\mathbf{I}\tilde f
\big|_{L^2_\dbF(\Om;H^1(0,T;
L^2(G)))}^{2}+\gamma |u^{\ast
}|_{\cH_T}^{2}\\
\ns&\ds =\big| \tilde f \big|_{L^2_\dbF(0,T;
L^2(G))}^{2}+\gamma |u^{\ast }|_{\cH_T}^{2}.
\end{array}
\end{equation}
By applying Theorem \ref{observability} to
$\widetilde{u}_{\gamma }$, we obtain that
\begin{equation}\label{6.31-2}
\begin{array}{ll}\ds
|\widetilde{u}_{\gamma
}|_{H^1_\dbF(0,T;L^2(G))\cap
L^2_\dbF(0,T;H^1(G))}
\\
\ns\ds \leq C \big(|h_{1}-h_{1}^{\ast
}|_{L^2_\dbF(0,T;H^1(\G))
}+|h_{1,t}-h_{1,t}^{\ast
}|_{L^2_\dbF(0,T;L^2(\G)) }+|h_{2}-h_{2}^{\ast
}|
_{L^2_\dbF(0,T;L^2(\G)) }\\
\ns\ds\qq +|L \widetilde{u}_{\gamma
}|_{L^2_\dbF(\Om;H^1(0,T; L^2(G)))}^{2}.
\end{array}
\end{equation}
This, together with \eqref{6.34}, implies the
estimate (\ref{6.31})  immediately.
\end{proof}

\section{Numerical Approximations}
In this section, we numerically solve {\bf Problem (C)} by the proposed method be given in section \ref{sec:6.3}.

\subsection{Algorithm description}
For the sake of simplicity, we set $(b_{ij})=I, b_1=b_2=b_3=0, b_4=1, T=1$ in this section. The Cauchy data are obtained by numerically solving the following initial-boundary value problem of stochastic hyperbolic equation
\begin{equation}
\begin{cases}
dz_t-\Delta zdt=fdt+z dW(t),& \mbox{in}\quad Q,\\
z=h_1,&\mbox{on}\quad  \Sigma,\\
z|_{t=0}=z_0(x),\ \frac{\partial z}{\partial t}|_{t=0}=\tilde{z}_0(x),&\mbox{in}\quad G.
\label{fp}
\end{cases}
\end{equation}
If the spatial domain $G$ is an interval or a rectangular region in the plane, we use the finite difference method to solve the problem, with time discretized via the Euler-Maruyama method. For irregular domains $G\subset\dbR^2$, we apply a meshless method using radial basis functions as kernels for the steady system at each temporal grid point after discretizing the equation in the temporal domain.

Drawing inspiration from \cite{Dalang2009},  we borrow the idea of two-step method of particular solution-method of fundamental solution \cite{ZCD} and address the minimization problem \eqref{6.28} using the kernel-based learning theory. 
By the superposition principle,
the stochastic hyperbolic equations \eqref{fp} can be  written as
\begin{equation}
z(x,t)=z^h(x,t)+z^p(x,t),
\end{equation}
where $z^p$ is a particular solution of the equation 
\begin{equation}
z^p_{tt}-\Delta z^p=f, \quad  \quad (x,t)\in Q,
\label{eq2}
\end{equation}
and 
$u^h$ satisfies
\begin{equation}\label{eq1}
\left\{\begin{aligned}\ds &dz^h_t-\Delta z^hdt=(z^h+z^p)dW(t), \quad (x,t)\in Q,\\
\ns\ds &z^h(x,t)=h_1-z^p(x,t),\quad \frac{\partial z}{\partial \nu}(x,t)=h_2-\frac{\partial z^p}{\partial \nu}(x,t),\quad (x,t)\in\Sigma.\end{aligned}\right.
\end{equation}\label{bc}
Equations \eqref{eq2}  and \eqref{eq1} can be seen as the non-homogeneous hyperbolic equation and  slight perturbations of the homogeneous hyperbolic equation, respectively. 

Now we solve {\bf Problem (C)} by above two problems. Let $\psi(r,t)\=\sqrt{1+c^2(r^2+t^2)}$ be the time-dependent radial basis function, where $c$ is the parameter to be determined, and
\begin{equation}
\tilde{z}^p(x,t)=\sum_{j=1}^N\zeta_j\psi_j(x,t),\quad (x,t)\in Q,
\label{zp}
\end{equation}
subject to \eqref{eq2}.
Here $\zeta_j$ are unknown coefficients to be determined, $N$ is the number of source points, and $\psi_j(x,t)$ are the basis functions given by
$
\psi_j(x,t)=\psi(x-\xi_j,t-\eta_j),j=1,\cdots,N.
$
From the theory of reproduced kernel Hilbert space, we know that
$$
z^p(x,t)\thickapprox\tilde{z}^p(x,t),\quad (x,t)\in\Sigma.
$$ 

We apply the idea of kernel-based learning theory, with Green's function $G(x,t)$ as kernel, to simulate the numerical approximation of problem \eqref{eq1}. However, the kernel lacks strong regularizing properties and its irregularity increases with the dimension. Specifically, the Green's function for the wave equation depends on the spatial domain's dimension \cite{Folland}. For $n=1$, 
$G(x,t)=\frac121_{|x|<t}$
is a bounded but discontinuous function. For $n=2$,
$G(x,t)=\frac{1}{\sqrt{2\pi}\sqrt{t^2-|x|^2}}1_{|x|<t}$
is unbounded and discontinuous.
For $n=3$, the Green's function is not a function, but a measure, requiring a specific definition for its convolution with a test function $\iota$:
$$
\begin{aligned}
(G * \iota)(x,t) & =\frac{1}{4 \pi} \int_0^t d s \int_{\partial B(0, s)} \iota(t-s, x-y) \frac{\sigma_s(d y)}{s} \\
& =\frac{1}{4 \pi} \int_0^t d s\ s \int_{\partial B(0,1)} \iota(t-s, x-s y) \sigma_1(d y) .
\end{aligned}
$$
For $n>3$, the Green's function becomes even more irregular. Therefore we only study examples in $n=1,2$ spatial domains, as more technical algorithms should be exploited for $n\geq3$. 
We approximate the mild solution $z^h(x,t)$ by 
\begin{equation}
\tilde{z}^h(x,t)=\sum_{j=1}^N \lambda_j\phi_j(x,t),
\label{s}
\end{equation}
where $\lambda_j$ are unknown coefficients to be determined, and $\phi_j(x,t)$ are the basis functions given by
\begin{equation}
\phi_j(x,t)=G(x-\xi_j,t-\eta_j),\ j=1,\cdots,N,
\end{equation}
where $\{(\xi_j,\eta_j)\}_{j=1}^N$ are the source points. Therefore, the solution of \eqref{fp} can be approximated by
\begin{equation}\label{diseq}
\tilde{z}(x,t)=\tilde{z}^h(x,t)+\tilde{z}^p(x,t)=\sum_{j=1}^N\lambda_j\phi_j(x,t)+\sum_{j=N+1}^{2N}\zeta_j\psi_j(x,t).
\end{equation}
Combining \eqref{zp}, \eqref{s} with \eqref{diseq}, {\bf Problem (C)} can be discretely written as a system of matrix equations:
\begin{equation}
\left\{\begin{aligned}
\ds &\sum_{j=1}^N\lambda_j\psi_j(x_i,y_i)=f(x_i,t_i),&i=1,\cdots,N_i\\
\ns\ds &\sum_{j=1}^N\lambda_j\psi_j(x_k,y_k)+\sum_{l=1}^N\zeta_l\phi_l(x_k,t_k)=h(x_k,t_k),&k=1,\cdots,N_b\\
\end{aligned}\right.
\end{equation}\label{diseq2}
where $h=(h_1,h_2),(x,t)\in\Sigma$ represents a vector of boundary conditions, $N_i$ is the number of interior points and $N_b$ is the number of boundary points.
The approximated solution should satisfy the equation ({\ref{diseq2}}), and the coefficient $(\lambda_1,\cdots,\lambda_N,\zeta_1,\cdots,\zeta_N)$, are determined to ensure $\tilde{z}$ satisfies the Cauchy data. Consequently, we solve a linear system of algebraic equations $A\Lambda=b$ for the unknowns  $\Lambda=(\lambda_1,\cdots,\lambda_N,\zeta_1,\cdots,\zeta_N)^\top$ based on the above analysis to obtain an approximated solution.

The optimal rule for the location of source points remains an open problem. The choice of source points should ensure that the basis functions $\phi_j(x,t)$ are analytic in $G$, guaranteeing that $\tilde{z}$ is analytic in the domain $Q$, which is necessary for algorithm convergence. Since the Green's function is discontinuous at $x=t$, we occupy uniformly distributed points on $x=t+R$ and $x=t-R$ as the source points, where $R$ is to be determined. 



Compared to the functional given in \eqref{6.28}, we compute $Lu-\mathbf{I}f$ using $A\Lambda-b$. In our computation we adapt the standard Tikhonov regularization to solve the matrix equation $A\Lambda=b$. Specifically, we solve the least squares problem $\ds\min_{\lambda}\lbrace \|A\Lambda-b\|^2+\gamma^2\|\Lambda\|^2\rbrace$ with a suitable choice of the regularization parameter $\gamma$, using such as  $L-curve$ method or {\it generalized cross validation} (GCV) \cite{EG}, where $\|\cdot\|$ denotes the usual Euclidean norm. This yields the regularized solution $\Lambda_\gamma=(\lambda^*_1,\cdots,\lambda^*_N,\zeta^*_1,\cdots,\zeta^*_N)^\top$,
%
%
and thus the approximated solution $\tilde{z}_\gamma^*$ is
\begin{equation}
\tilde{z}_\gamma^*(x,t)=\sum_{j=1}^N\lambda_j^*\phi_j(x,t)+\sum_{j=N+1}^{2N}\zeta_j^*\psi_j(x,t).
\end{equation}
The numerical results presented in the following section demonstrate that the proposed scheme is both feasible and efficient. To illustrate the comparison between the exact solution and its approximation, $\tilde{z}_\gamma^*$, and in order to avoid the ``inverse crime", we compute $z(x,t)$ using a finite difference method and verify the efficiency of the proposed method by following defined relative errors:
\begin{equation*}
E_1(x,t)=\frac{|\mathbb{E}[z_\gamma^*(x,t)-z(x,t)]|}{|\mathbb{E}[z(x,t)]|},
\end{equation*}
\begin{equation*}
E_2(x)=\frac{\|z_\gamma^*(x,\cdot)-z(x,\cdot)\|_{L^2_{\mathcal{F}_t}(\Omega;L^2(0,T))}}{\|z(x,\cdot)\|_{L^2_{\mathcal{F}_t}(\Omega;L^2( 0,T))}},
\label{relative error}
\end{equation*}
and
\begin{equation*}
E_3(t)=\frac{\|z_\gamma^*(\cdot,t)-z(\cdot,t)\|_{L^2_{\mathcal{F}_t}(\Omega;L^2(G))}}{\|z(\cdot,t)\|_{L^2_{\mathcal{F}_t}(\Omega;L^2(G))}}.
\label{relative errort}
\end{equation*}
We assess the variations of error at each ``point" using the relative error $E_1(x,t)$ when $n=1$. For $n=2$, we illustrate the errors by plotting $E_2(x)$ and $E_3(t)$ to provide a clearer visual representation.



\subsection{Numerical examples}
We complete the verification of the proposed algorithm by following examples as in 1D and 2D spatial domains.

\begin{example}\label{ex1d}
Let $G=(0,1)$ and $\Gamma=\lbrace0\rbrace\cup\lbrace 1\rbrace$. Suppose that

(a)
$z_0(x)=e^x,\quad \tilde{z}_0(x)=e^x,\quad f(x,t)=0, \quad 
h_1(x,t)=\begin{cases}
e^t,&\quad x=0,\\
e^{1+t},&\quad x=1.
\end{cases}
$

%
%

(b)
$
z_0(x)=\sin{(\pi x)},\quad \tilde{z}_0(x)=-\sin{(\pi x)},\quad h_1(x,t)=0,\quad
f(x,t)=(1+\pi^2)\sin{(\pi x)}e^{-t}.
$


(c)
$z_0(x)=\begin{cases}
\frac{10}{7}x,&0\leq x< 0.7\\
\frac{10}{3}(1-x),&0.7\leq x\leq 1
\end{cases},\quad \tilde{z}_0(x)=0,\quad f(x,t)=0,\quad h_1(x,t)=0.$
\end{example}

Figure \ref{ex1d1} shows the numerical results for {\bf Problem (C)} with a noise level $\delta=3\%$ from different perspectives. As with studies of stochastic partial differential equations, it is important to check a large number of sample paths in numerical experiments to accurately simulate the expectation of the solution. We conducted tests with varying numbers of sample paths. From Figure \ref{ex1d1}, we observe that when the number of sample paths is fewer than 100, the numerical approximations are somewhat rough. The results become smoother as the number of sample paths increases. However, interestingly, when the number of sample paths exceeds100, the results do not improve significantly. Thus to save computational costs, we only do the experiments with number of sample paths $\# = 100$ in following numerical tests.

\begin{figure}[htbp]
  \centering
 \captionsetup[subfloat]{labelsep=none,format=plain,labelformat=empty}
\subfloat{\includegraphics[width=0.2\textwidth]{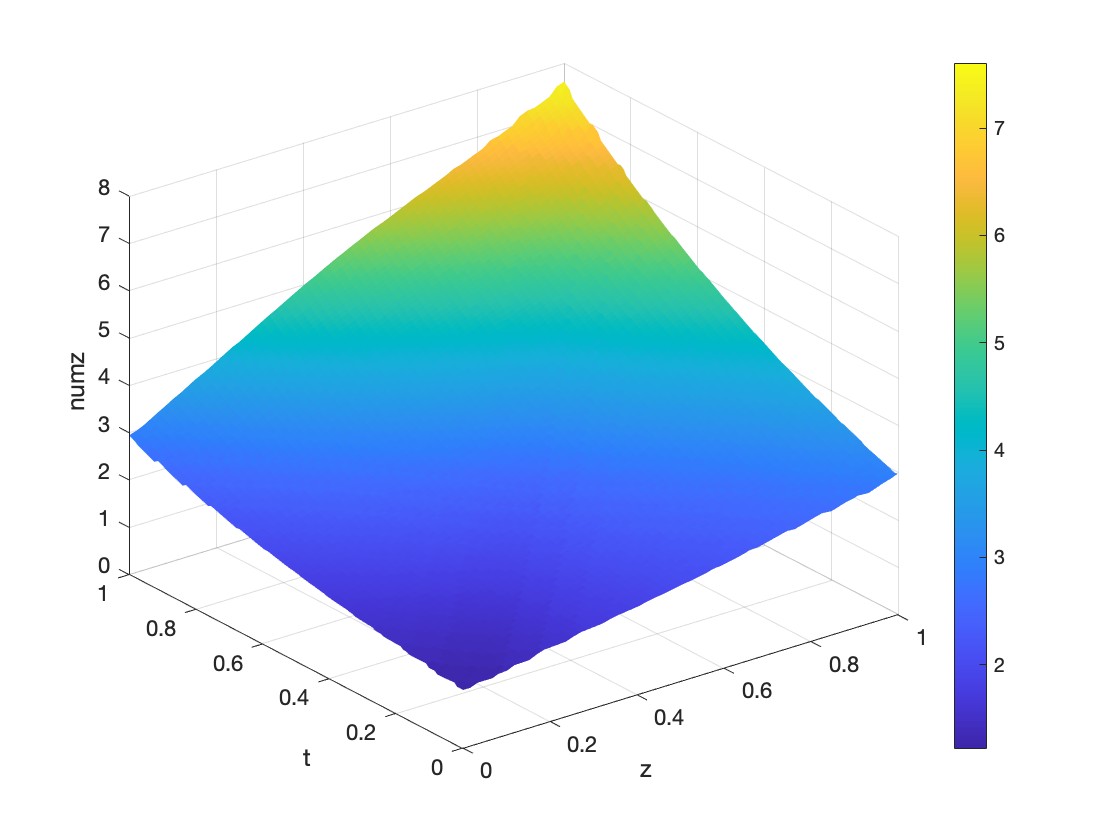}}
\subfloat{\includegraphics[width=0.2\textwidth]{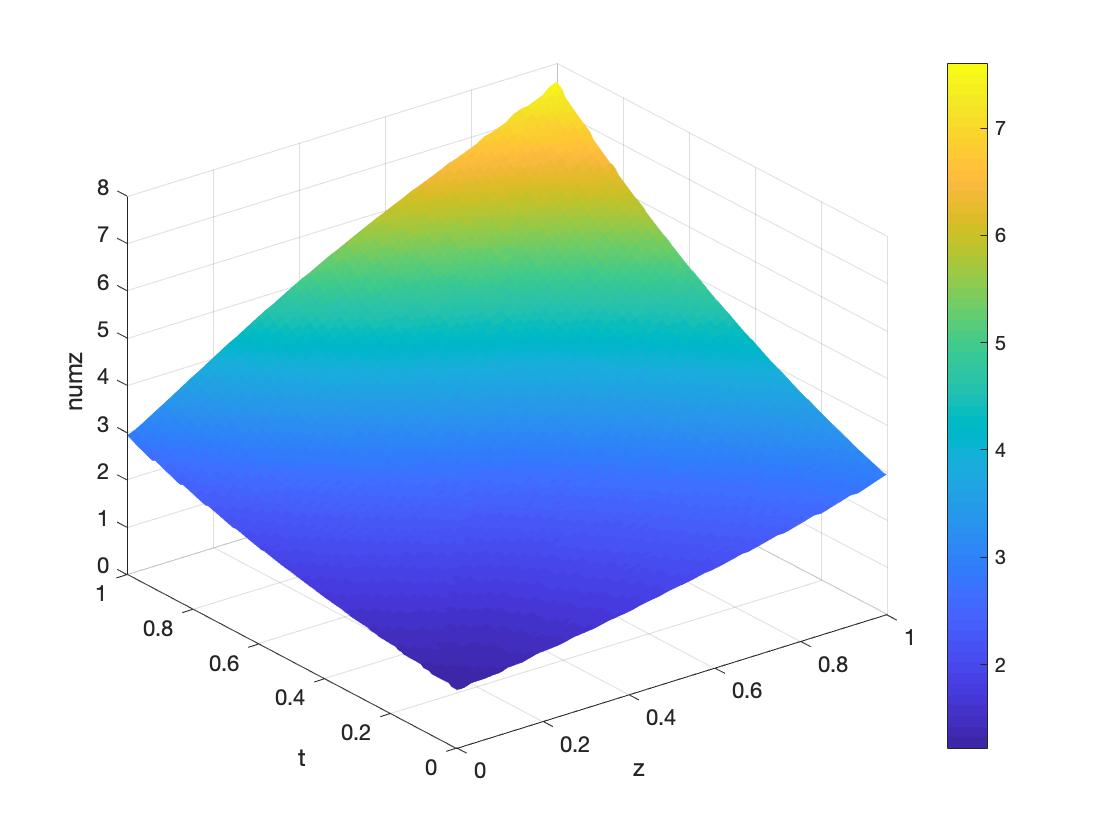}}
\subfloat{\includegraphics[width=0.2\textwidth]{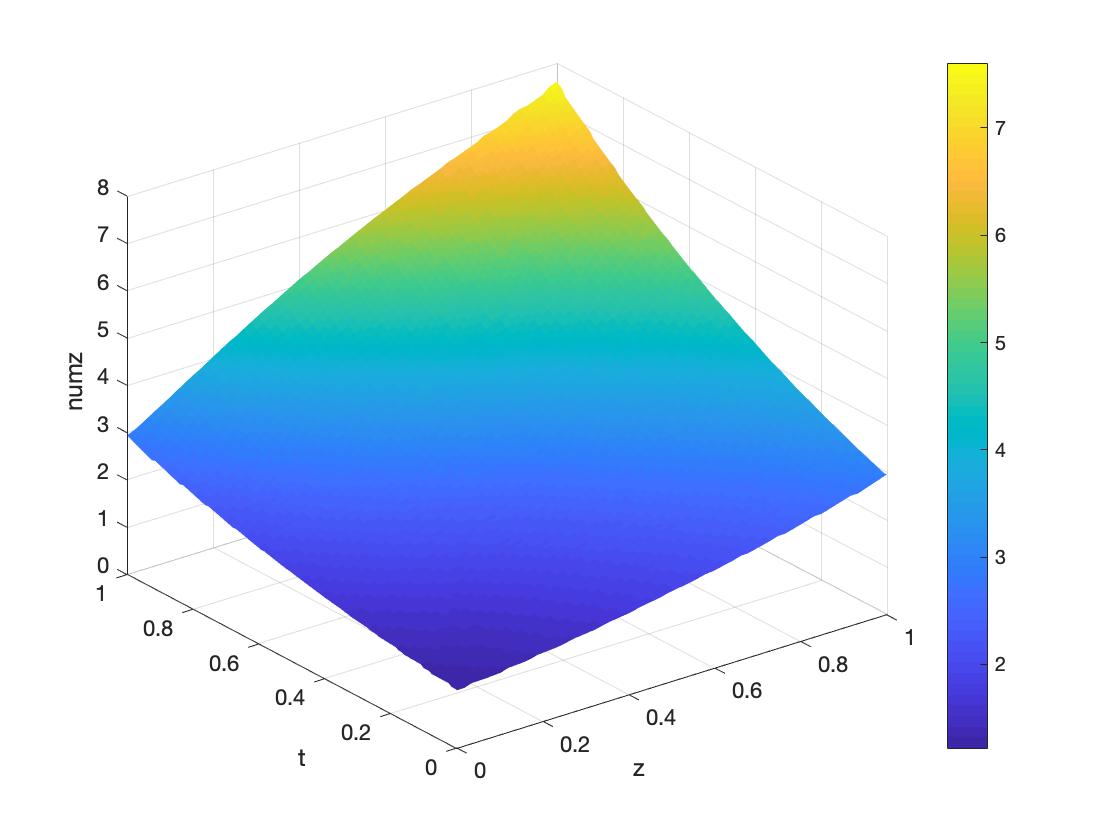}}
\subfloat{\includegraphics[width=0.2\textwidth]{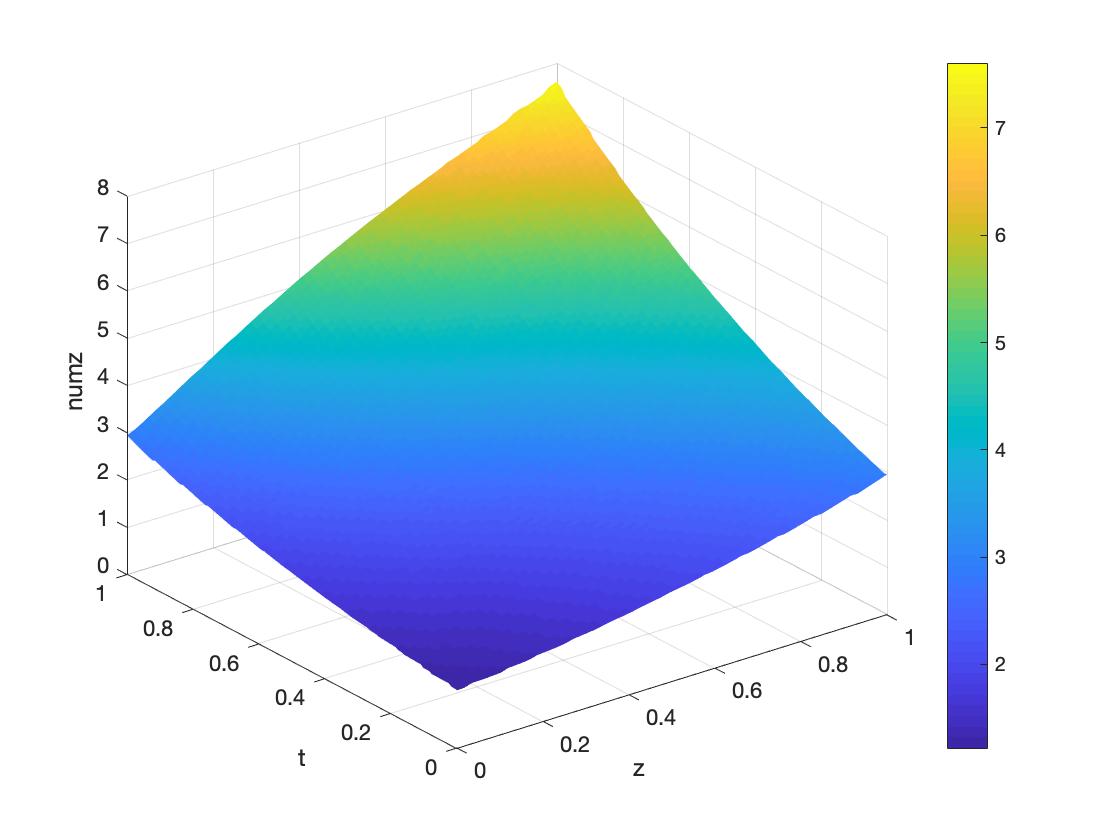}}
\subfloat{\includegraphics[width=0.2\textwidth]{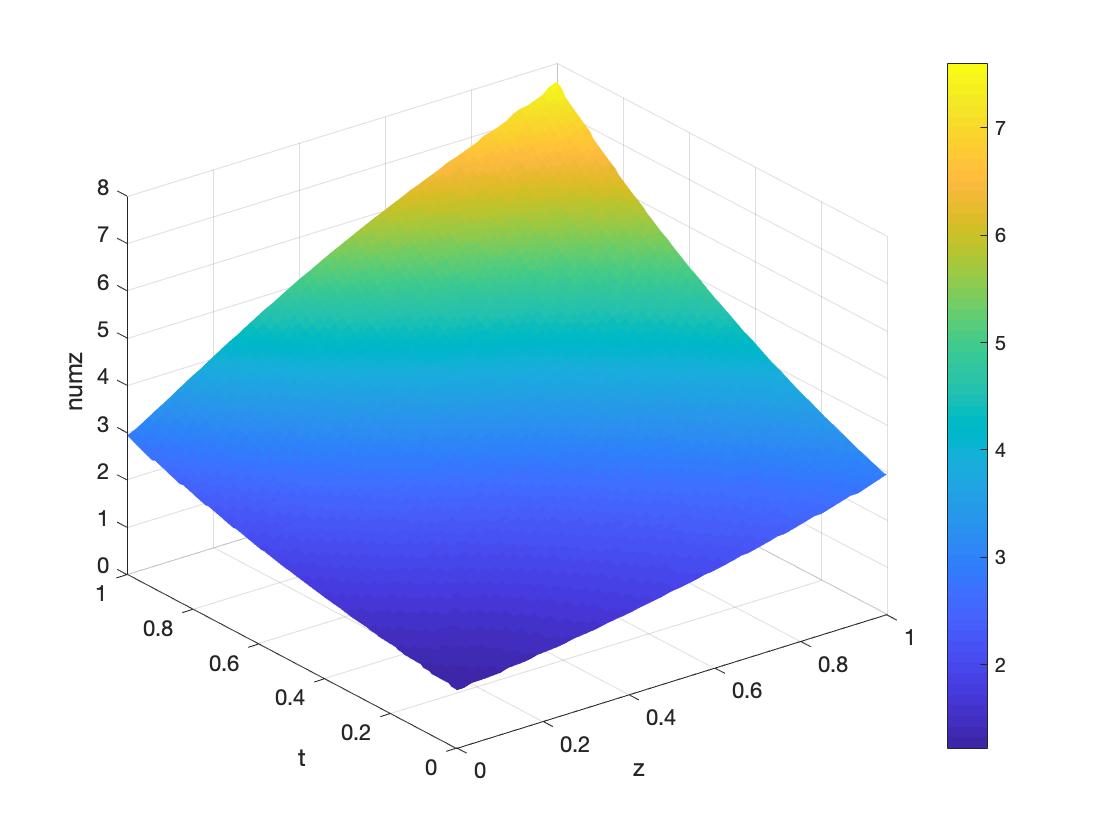}}\\
\subfloat[(a)$\#=1$]{\includegraphics[width=0.2\textwidth]{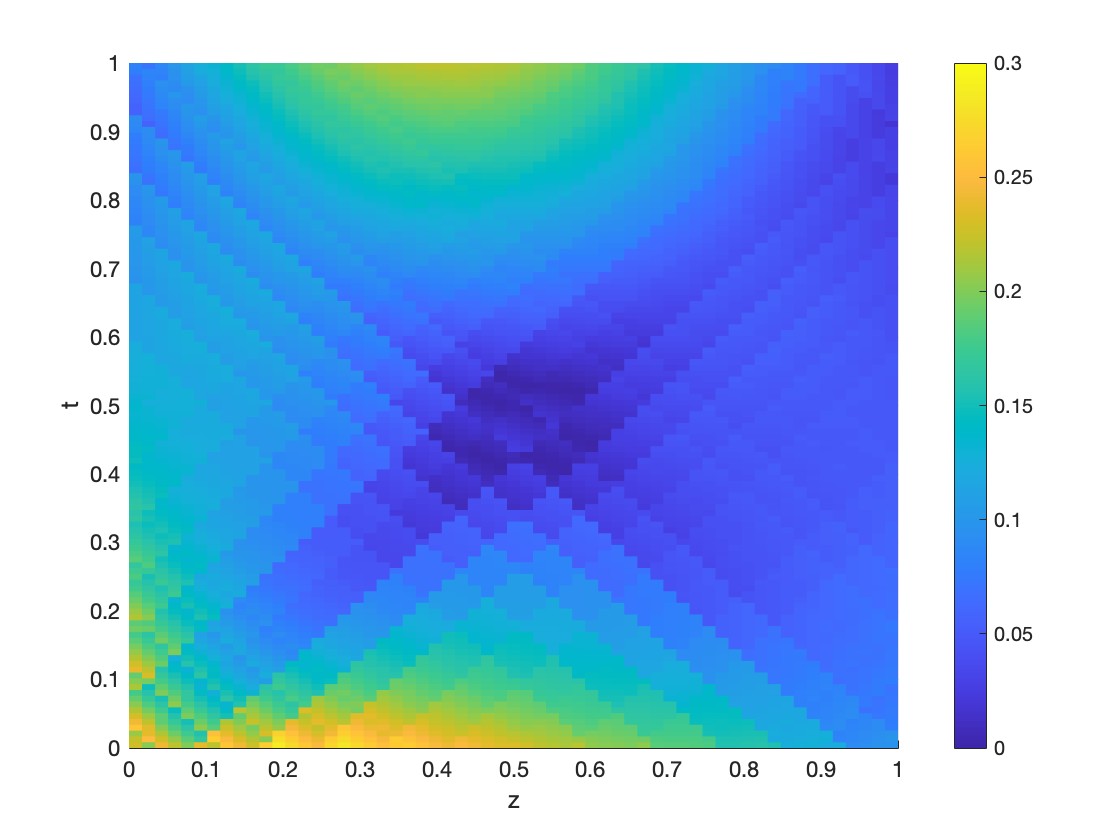}}
\subfloat[(b)$\#=10$]{\includegraphics[width=0.2\textwidth]{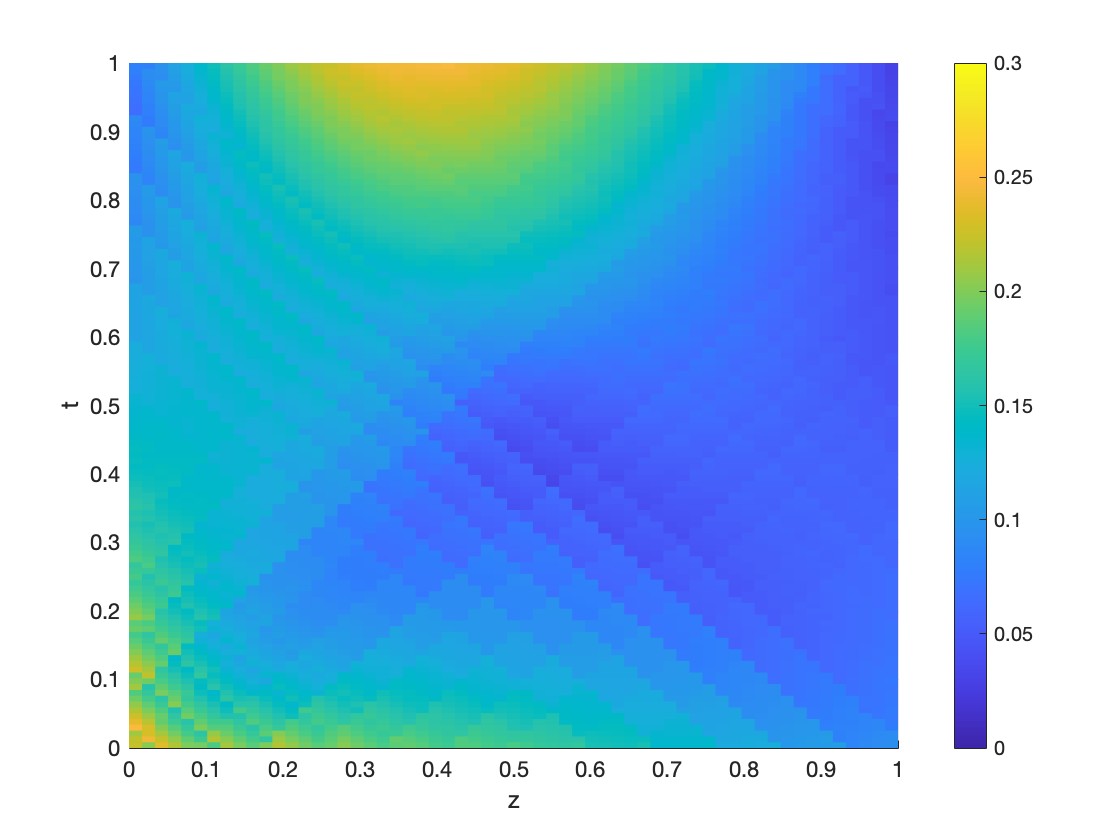}}
\subfloat[(c)$\#=100$]{\includegraphics[width=0.2\textwidth]{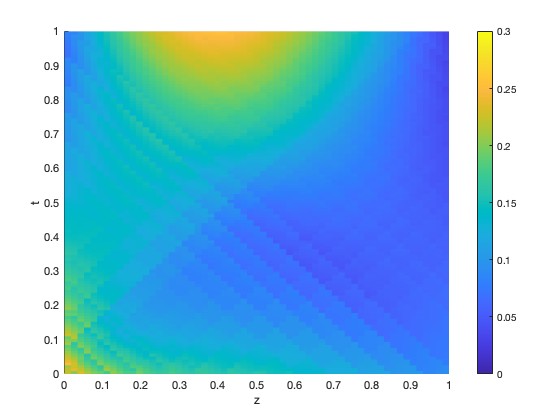}}
\subfloat[(d)$\#=1000$]{\includegraphics[width=0.2\textwidth]{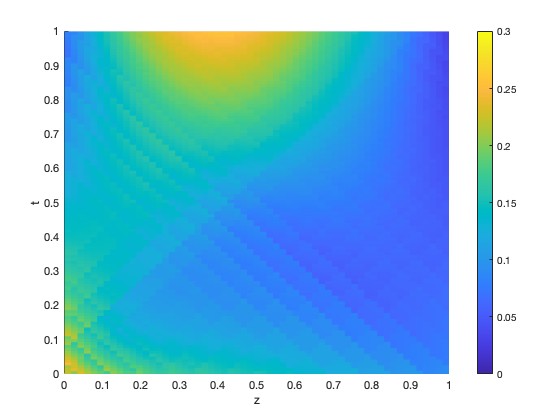}}
\subfloat[(e)$\#=10000$]{\includegraphics[width=0.2\textwidth]{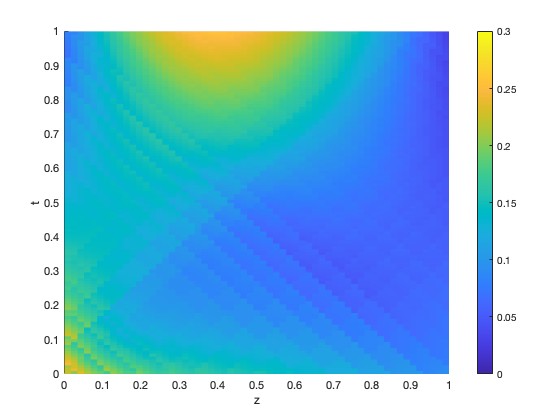}}
\caption{Numerical approximations (above) and relative errors (below) of Example \ref{ex1d}(a) from different perspectives and relative errors $E_1(x,t)$ of different numbers of sample paths $\#=1,10,100,1000,10000$ with  $\delta=3\%$.}\label{ex1d1}
\end{figure}

We test the optimal choices of parameters $R$ and $c$ when $\delta=3\%$ by Example \ref{ex1d}(b). Firstly, we fix $R=0.1$ and observe the relative error $E$ with the change of $c$. It can be seen from Figures  \ref{rex1d3}-\ref{ex1d3} that $c=0.6$, $R=1.5$ perform well. Thus, we fix $R=1.5$ and $c=0.6$ in the numerical test of 1D examples.  We note that suitable choices of $c$ and $R$ is crucial for numerical simulation. While these parameters may not always be optimal for every example, they are sufficient to test the effectiveness of the proposed numerical algorithm. Further discussion on parameter choices is not provided here.

\begin{figure}[htbp]
  \centering
\subfloat[$E_2(x)$]{\includegraphics[width=0.3\textwidth]{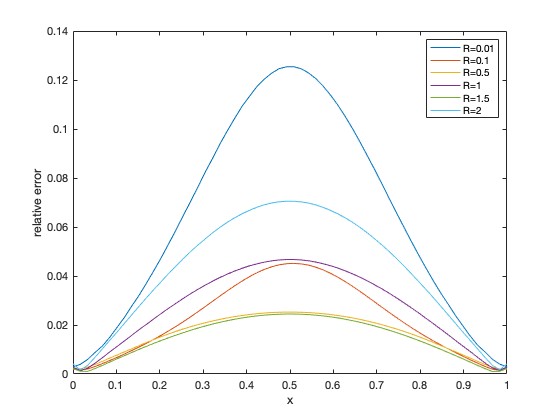}}\ 
\subfloat[$E_3(t)$]{\includegraphics[width=0.3\textwidth]{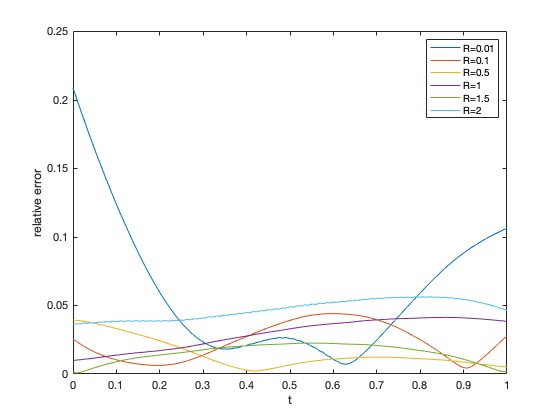}}
\caption{Relative errors with different $R$ of Example \ref{ex1d}(b).}\label{rex1d3}
\end{figure}

\begin{figure}[htbp]
  \centering
  \captionsetup[subfloat]{labelsep=none,format=plain,labelformat=empty}
\subfloat[\qquad \qquad (a)\ $c=0.5$]{\includegraphics[width=0.2\textwidth]{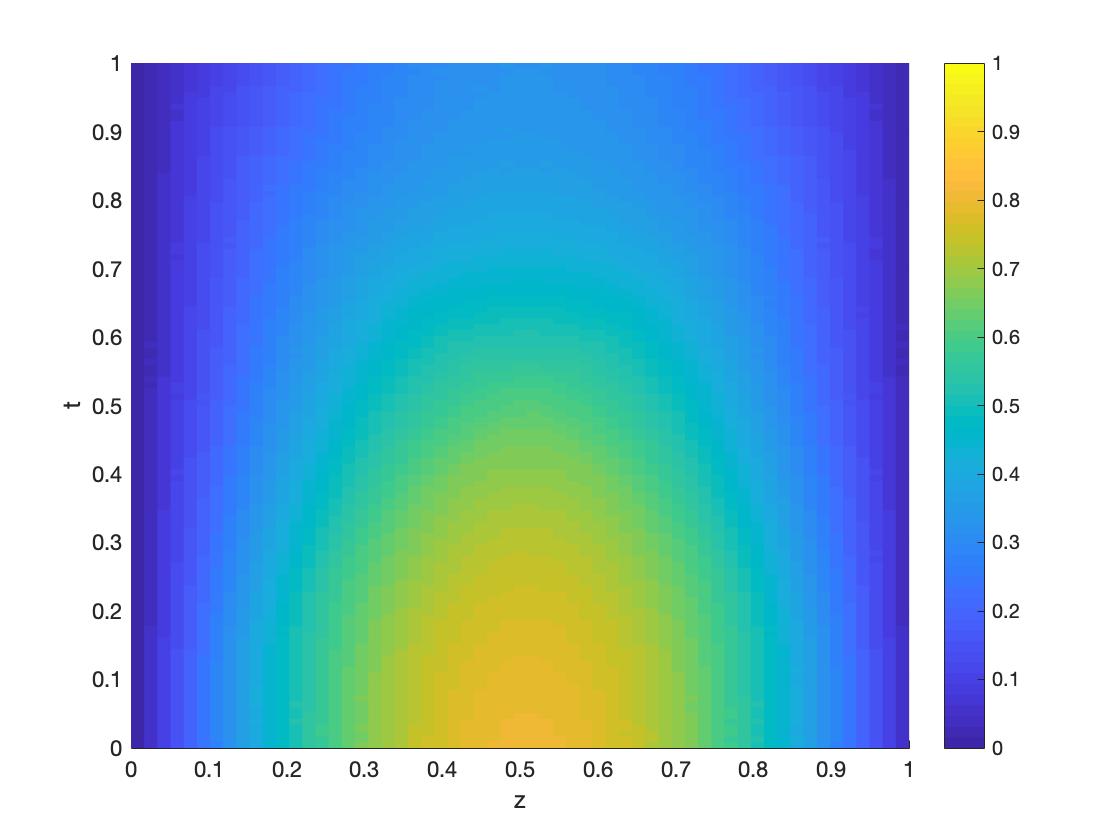}}\ 
\subfloat{\includegraphics[width=0.2\textwidth]{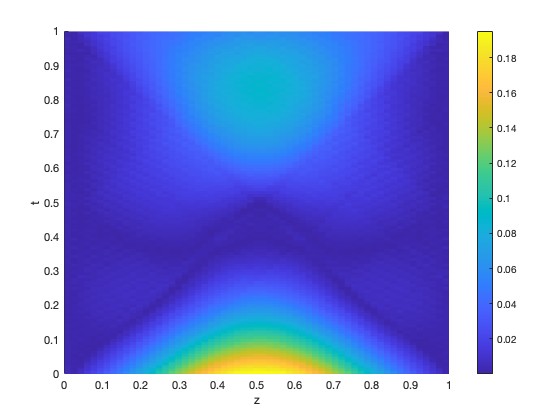}}\
\subfloat[\qquad \qquad (b)\ $c=0.6$]{\includegraphics[width=0.2\textwidth]{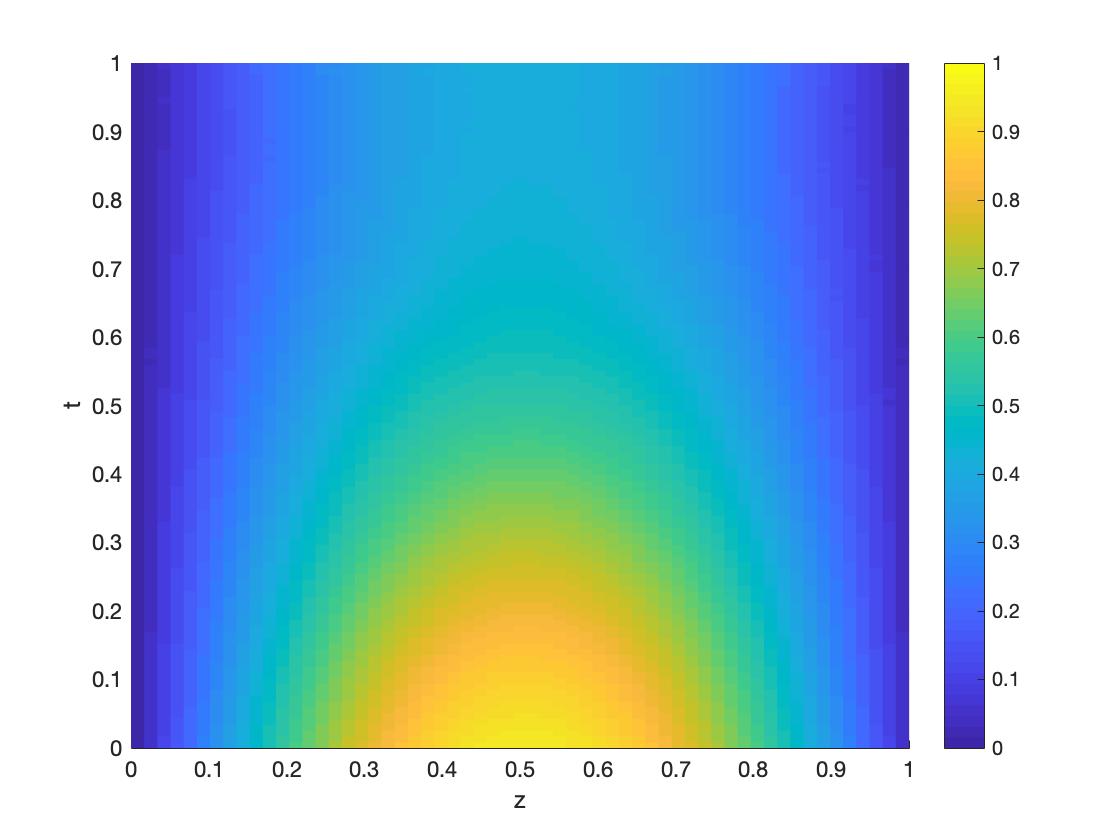}}\ 
\subfloat{\includegraphics[width=0.2\textwidth]{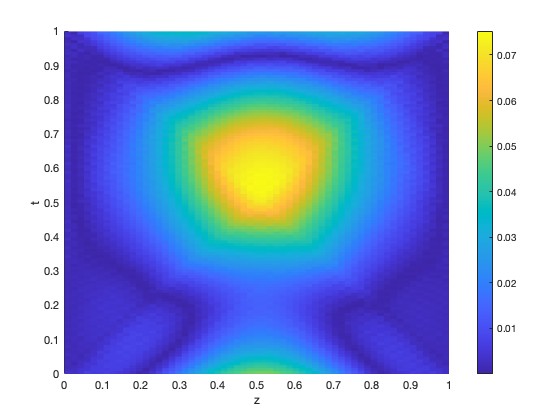}}\\
\subfloat[\qquad \qquad (c)\ $c=0.7$]{\includegraphics[width=0.2\textwidth]{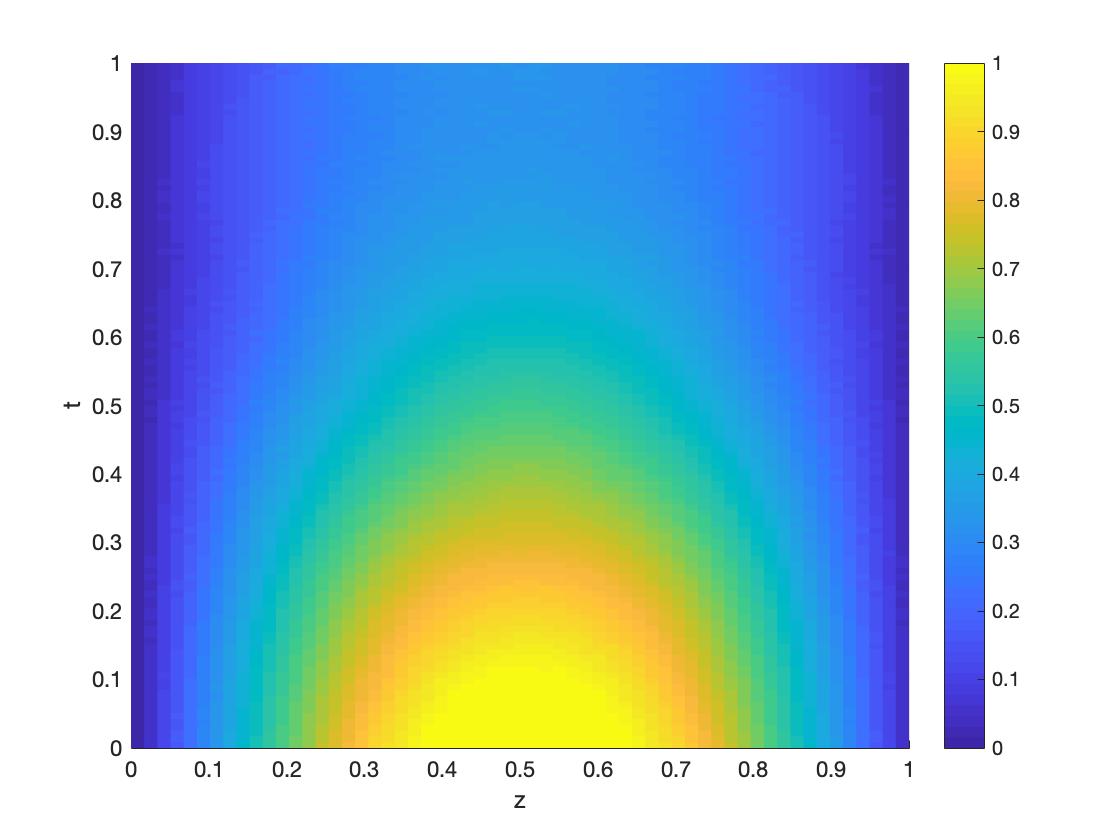}}\ 
\subfloat{\includegraphics[width=0.2\textwidth]{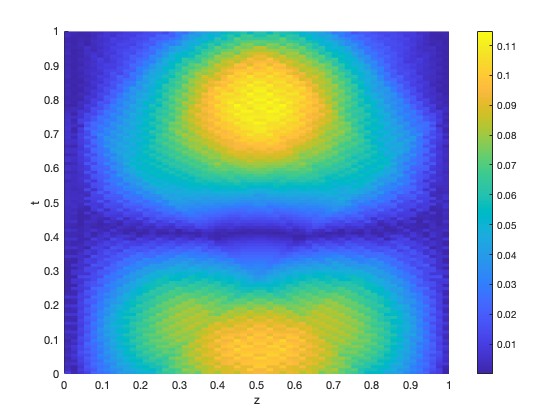}}\ 
\subfloat[\qquad \qquad (d)\ $c=0.8$]{\includegraphics[width=0.2\textwidth]{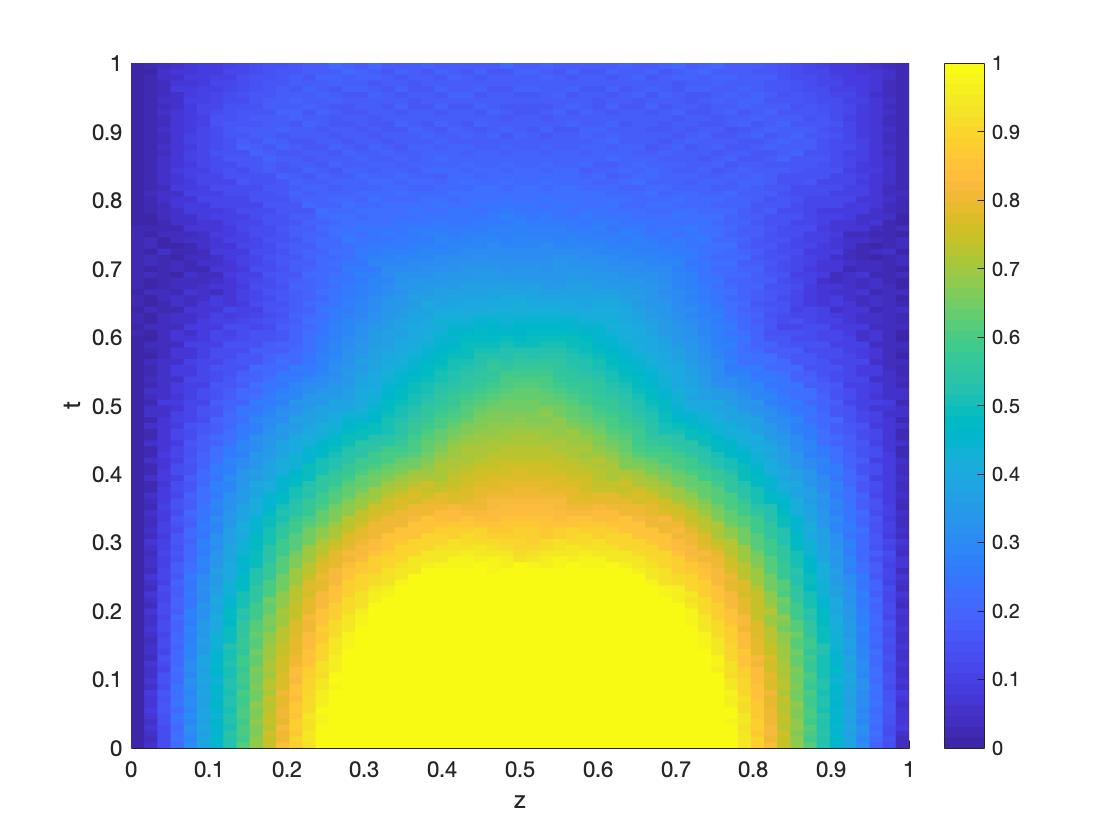}}\ 
\subfloat{\includegraphics[width=0.2\textwidth]{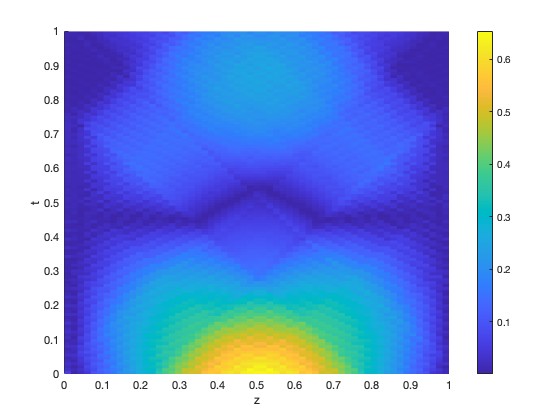}}\\
\subfloat[\qquad \qquad (e)\ $c=0.9$]{\includegraphics[width=0.2\textwidth]{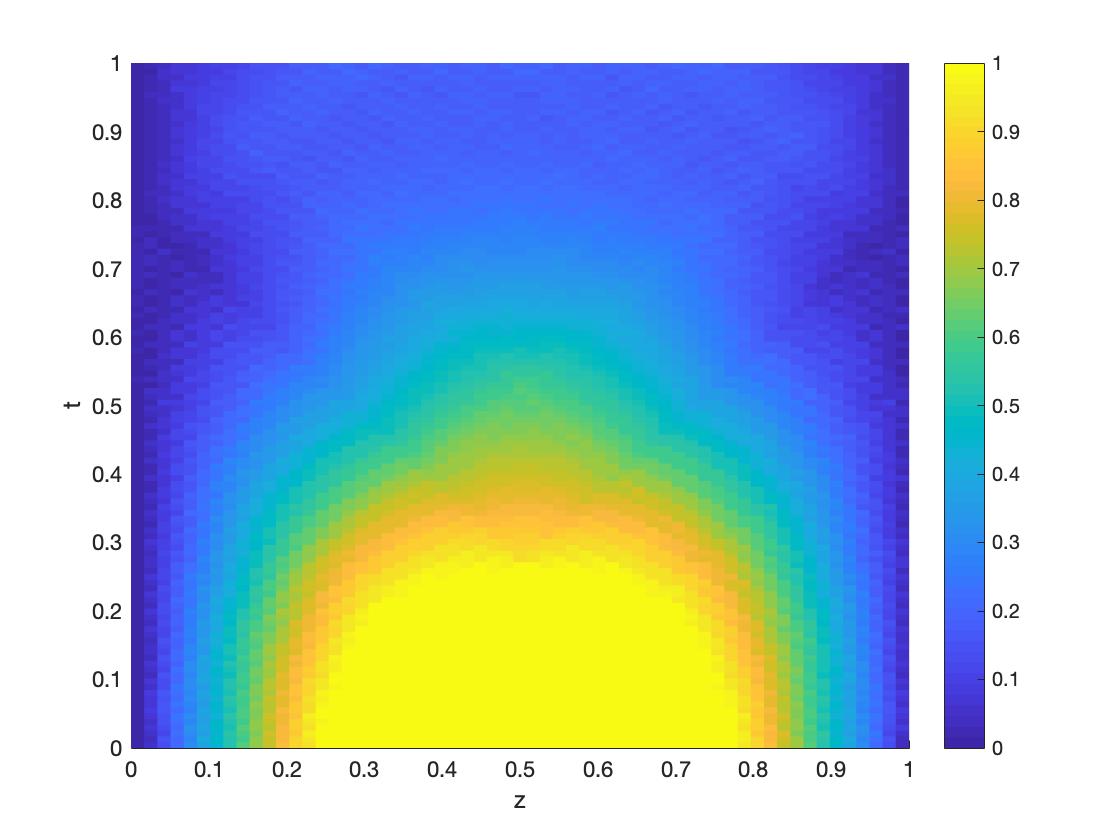}}\ 
\subfloat{\includegraphics[width=0.2\textwidth]{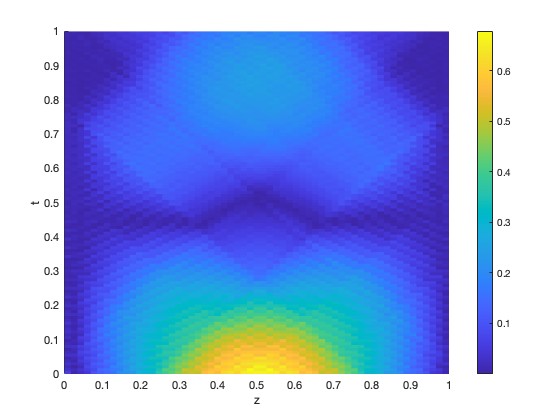}}\ 
\subfloat[\qquad \qquad (f)\ $c=1$]{\includegraphics[width=0.2\textwidth]{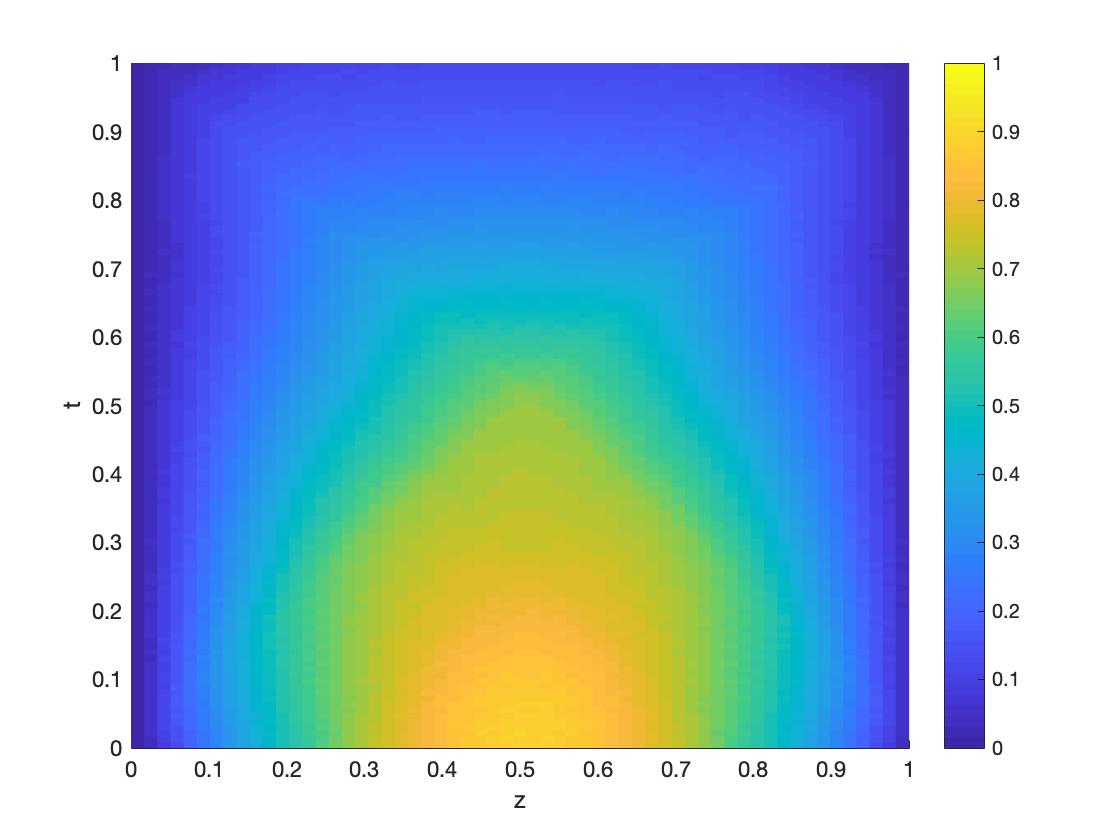}}\ 
\subfloat{\includegraphics[width=0.2\textwidth]{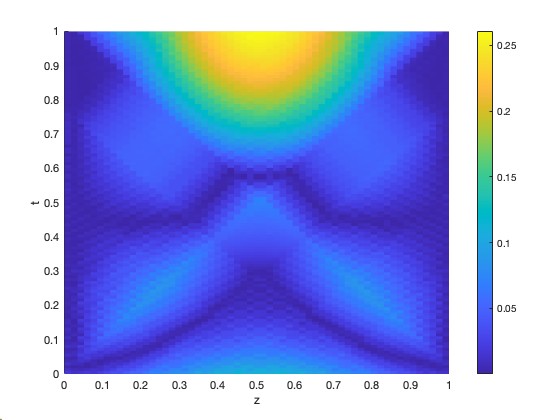}}
\caption{Numerical approximations (left) and relative errors $E_1(x,t)$ (right) by expectation of solutions with different parameter $c$ of Example \ref{ex1d}(b).}\label{ex1d3}
\end{figure}

For Example \ref{ex1d}(c), we discuss how the noise level affect the results, as shown in Figure \ref{ex1d5}.
%

\begin{figure}[htbp]
  \centering
    \captionsetup[subfloat]{labelsep=none,format=plain,labelformat=empty}
\subfloat{\includegraphics[width=0.3\textwidth]{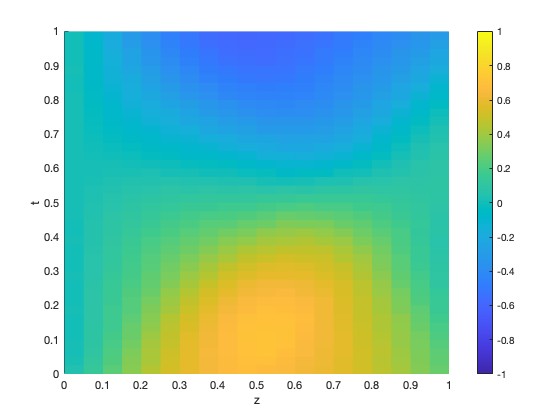}}\ 
\subfloat{\includegraphics[width=0.3\textwidth]{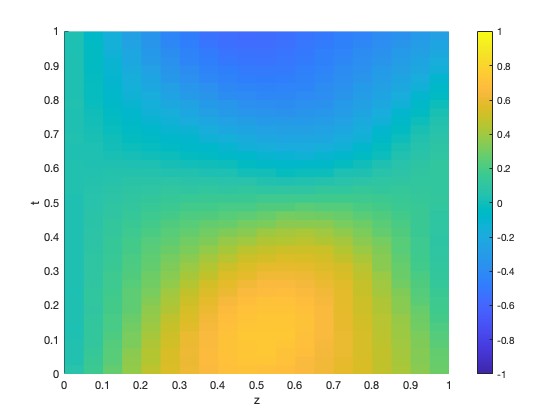}}\ 
\subfloat{\includegraphics[width=0.3\textwidth]{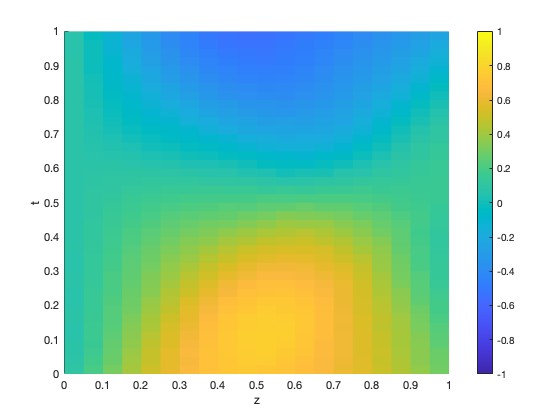}}\\
\subfloat[(a)\ $\delta=1\%$]{\includegraphics[width=0.3\textwidth]{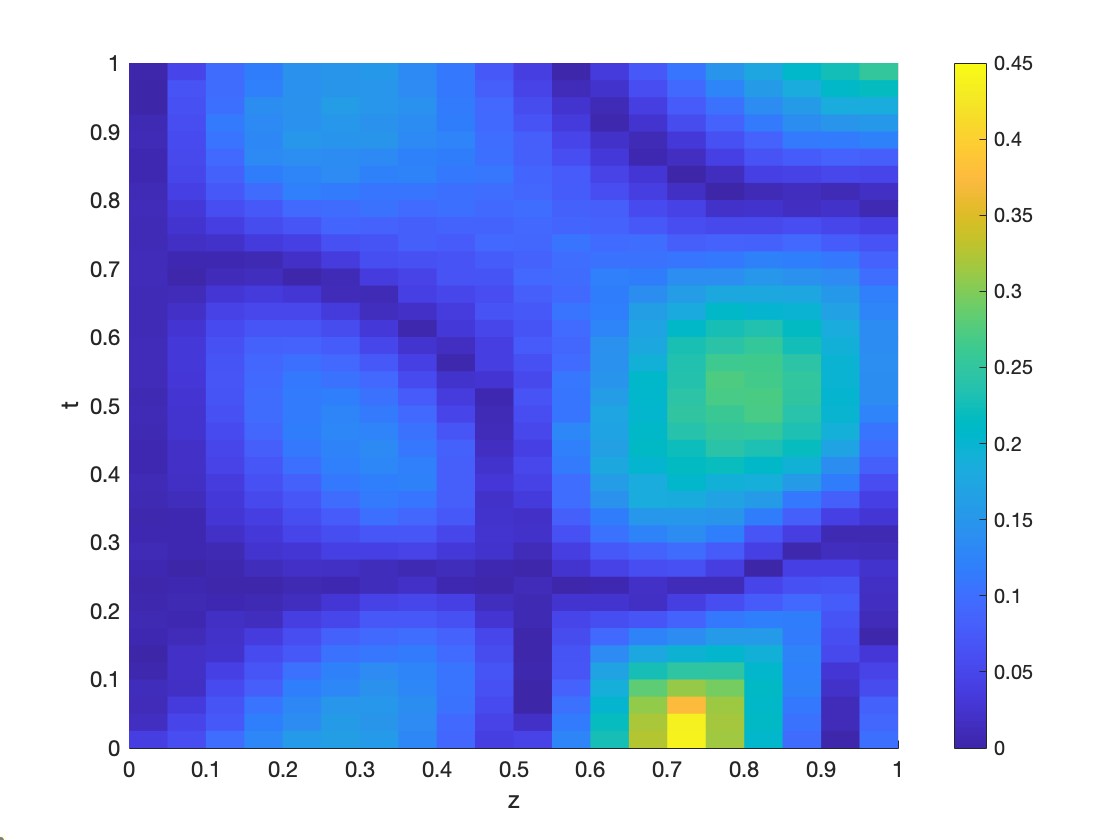}}\ 
\subfloat[(b)\ $\delta=5\%$]{\includegraphics[width=0.3\textwidth]{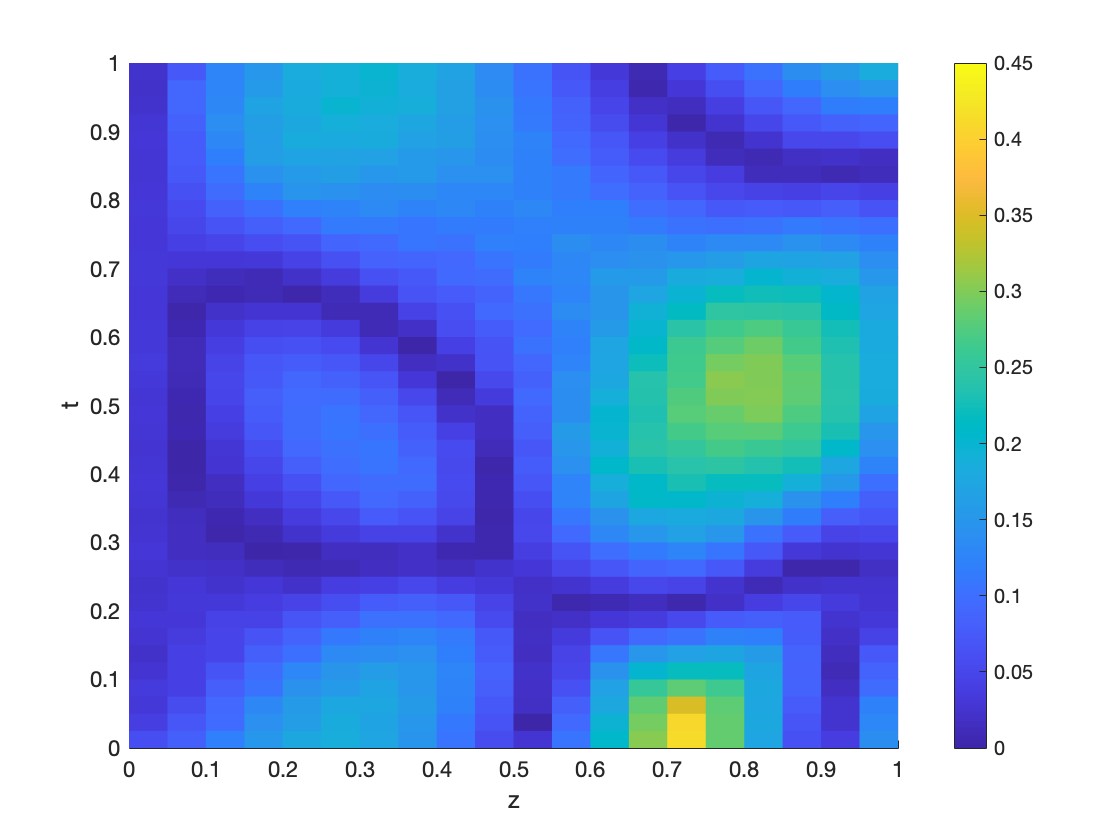}}\ 
\subfloat[(c)\ $\delta=10\%$]{\includegraphics[width=0.3\textwidth]{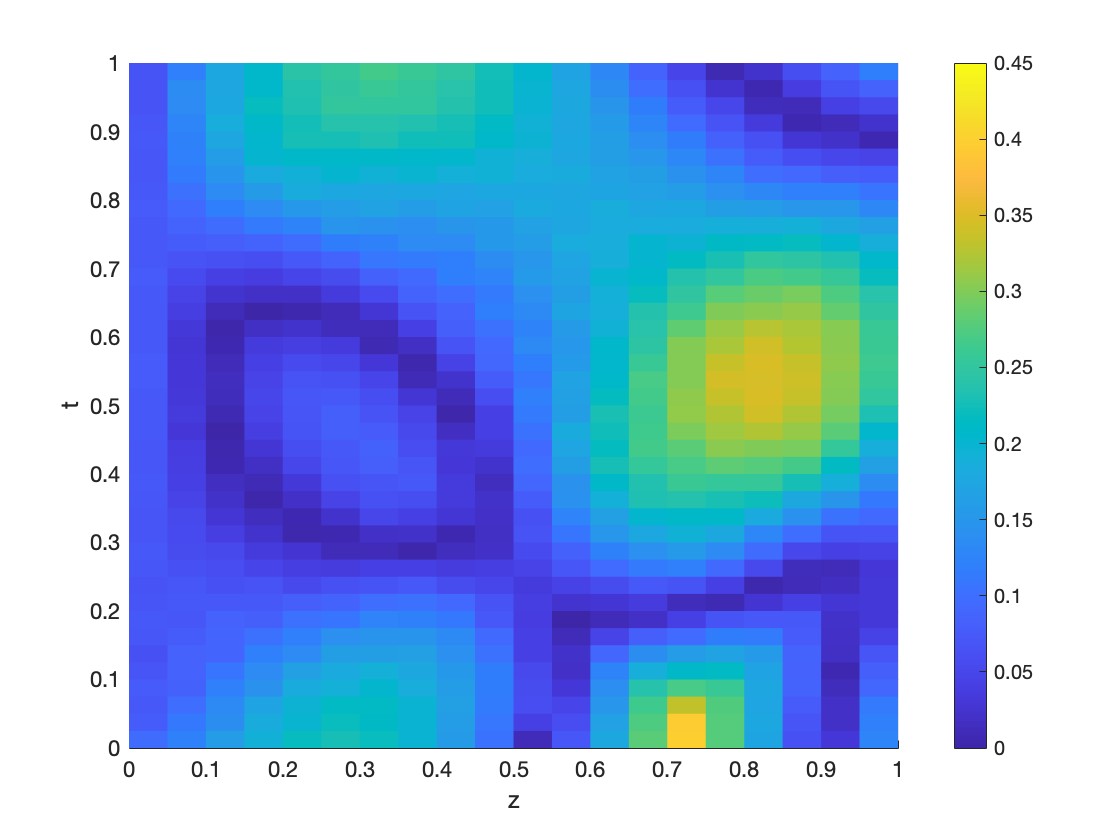}}
\caption{Numerical approximations (above) and relative errors  $E_1(x,t)$ (below) with different noise levels of Example \ref{ex1d}(c).}\label{ex1d5}
\end{figure}

In the following example, we probe numerical algorithms in $G\subset\dbR^2$. 
Due to the storage limitations of computers and computational costs, we restrict the number of sample paths $\#=10$ in Example \ref{ex2d} (a) and (b). 
 
\begin{example}\label{ex2d}

(a) Suppose $G=\{(x_1,x_2)\in\mathbb{\dbR}^2| x_1^2+x_2^2<1\}$.
%
%
%
Let  $z_0(x_1,x_2)=0$, $\tilde{z}_0(x_1,x_2)=0$,
$f(x_1,x_2,t)=2(1+\pi^2t^2)\sin\pi (x_1+ x_2)$, $h_1(x_1,x_2,t)=t^2\sin\pi( x_1+ x_2).$

\vspace{3pt}
(b) Suppose $G$ is bounded by the curve  $\ds r=\sin2\theta,\theta\in (0,\frac\pi2)$ in polar coordinates, as shown in Figure \ref{sr}.
Let $z_0(x_1,x_2)=0,$ $\tilde{z}_0(x_1,x_2)=0, $
$f(x_1,x_2,t)=2(1+\pi^2t^2)\cos\pi x_1\cos\pi x_2$, $h_1(x_1,x_2,t)=t^2\cos\pi x_1\cos\pi x_2$.
\begin{figure}[htbp]
  \centering
\subfloat[]{\includegraphics[width=0.3\textwidth]{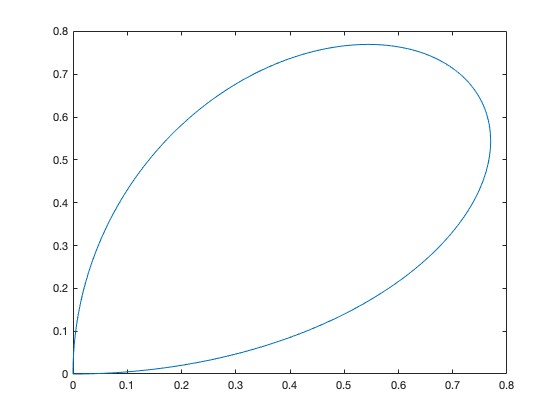}}
\caption{Domain G of Example \ref{ex2d} (b).}\label{sr}
\end{figure}

\vspace{3pt}
(c) Suppose  $G=(0,1)^2$, $G_1=[0.4,0.6]^2\subset G$.
%
Let $z_0(x_1,x_2)=\begin{cases}
1,&(x_1,x_2)\in G_1,\\
0,&(x_1,x_2)\in G\backslash G_1,
\end{cases}$ $\tilde{z}_0(x_1,x_2)=0$, 
$f(x_1,x_2,t)=0$,  $h_1(x_1,x_2)=0.$
\end{example}


Figure \ref{rtex2d2} shows the the relative error $E_2(x_1,x_2)$ and $E_3(t)$ for different noise levels of Example \ref{ex2d}(a). 

\begin{figure}[htbp]
  \centering
\subfloat[$\delta=1\%$]{\includegraphics[width=0.24\textwidth]{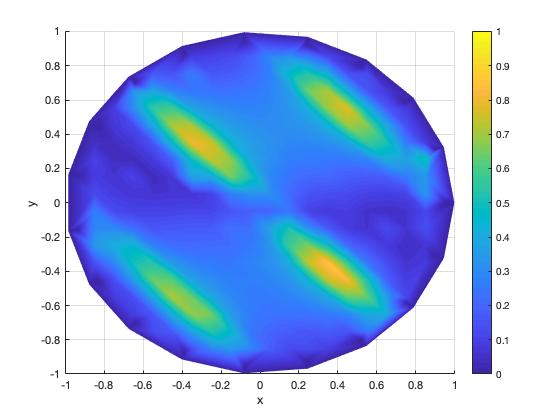}}
\subfloat[$\delta=5\%$]{\includegraphics[width=0.24\textwidth]{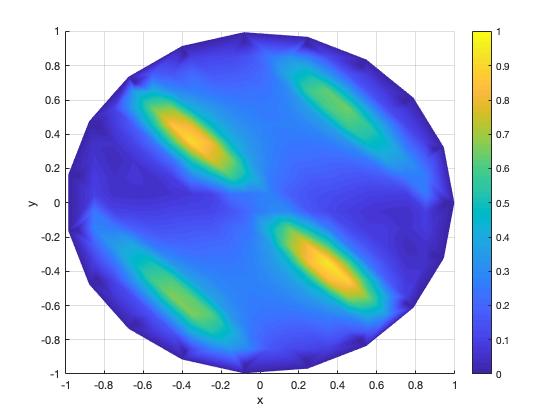}}
\subfloat[$\delta=10\%$]{\includegraphics[width=0.24\textwidth]{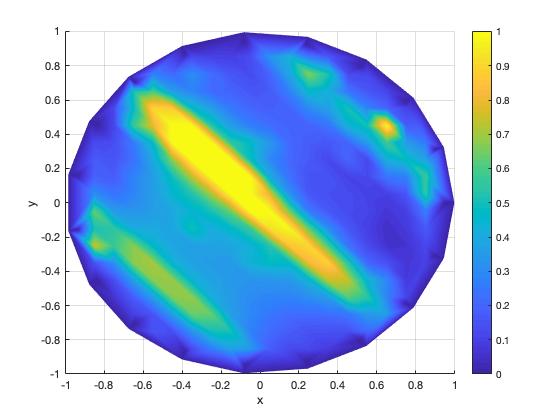}}
%
\subfloat[$E_3(t)$]{\includegraphics[width=0.24\textwidth]{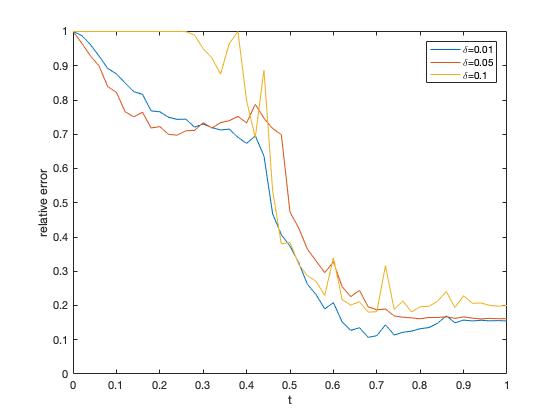}}
\caption{Relative errors $E_2(x_1,x_2)$ and  $E_3(t)$ for different noise level of Example \ref{ex2d}(a).}\label{rtex2d2}
\end{figure}

The influence of parameters $c$ and $R$ are explored by Example \ref{ex2d}(c) with $\delta=2\%$ in a leaf-shaped area. Firstly we set $R=1$ to test the possible range of values for $c$. Figure \ref{ex2d3}(a-e) show the relative error $E_2(x_1,x_2)$ for different values of $c$ when $R=1$. Subsequently, we tested the choice for $R$ by setting $c=2.5$ based on the previous results . Figure \ref{ex2d3}(f-j) indicate that $R=1.5$ is a suitable choice. We further verify the choice of these parameters by examining the relative errors $E_3(t)$ in Figure \ref{rex2d3}. Therefore, we consistently select $c=2.5, R=1.5$ in the subsequent computation.

\begin{figure}[htbp]
  \centering
\subfloat[$c=0.1$]{\includegraphics[width=0.2\textwidth]{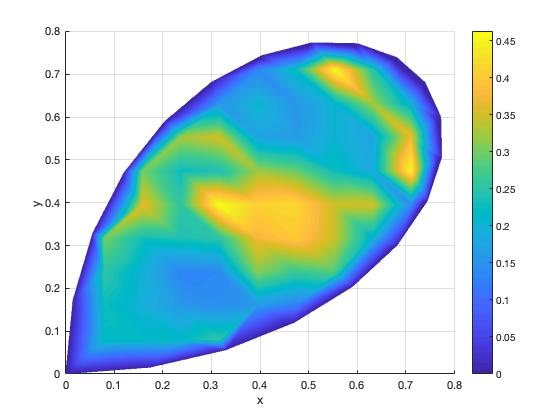}}
\subfloat[$c=0.5$]{\includegraphics[width=0.2\textwidth]{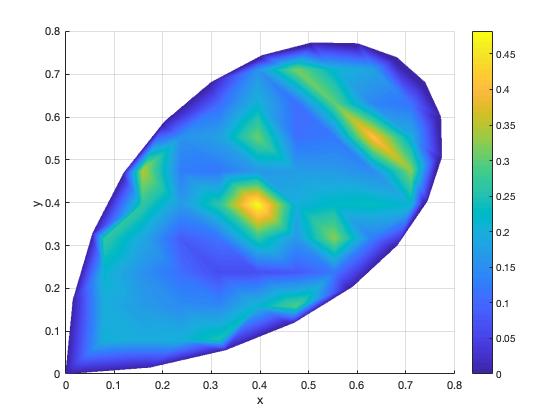}}
\subfloat[$c=1$]{\includegraphics[width=0.2\textwidth]{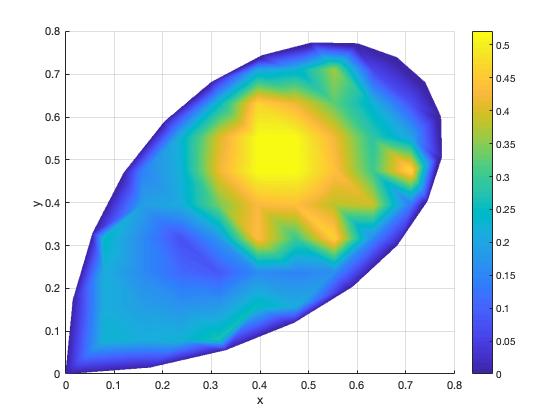}}
\subfloat[$c=1.5$]{\includegraphics[width=0.2\textwidth]{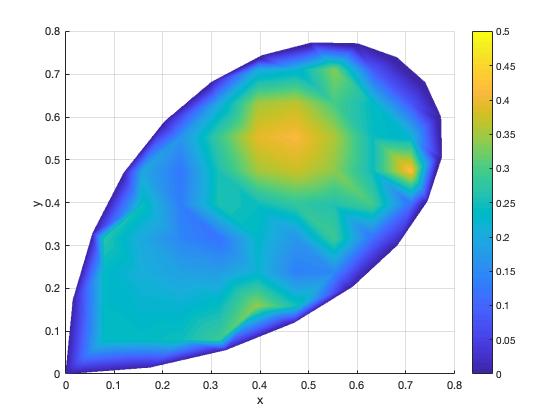}}
\subfloat[$c=2$]{\includegraphics[width=0.2\textwidth]{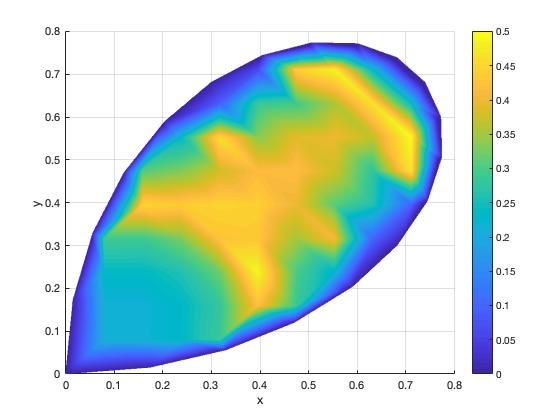}}
\\
\subfloat[$R=0.5$]{\includegraphics[width=0.2\textwidth]{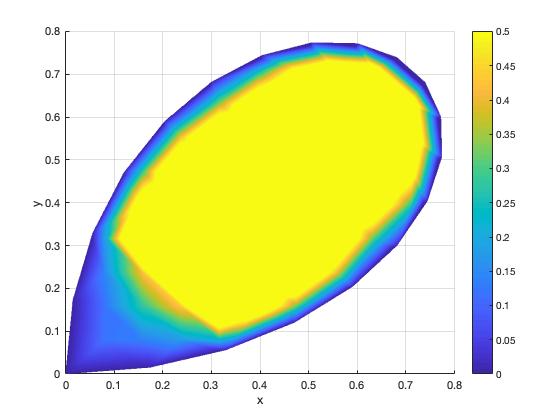}}
\subfloat[$R=1$]{\includegraphics[width=0.2\textwidth]{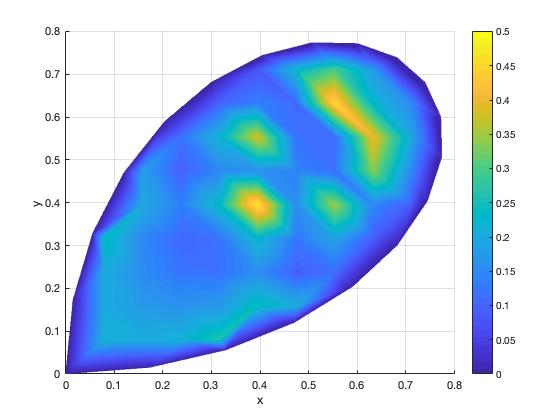}}
\subfloat[$R=2$]{\includegraphics[width=0.2\textwidth]{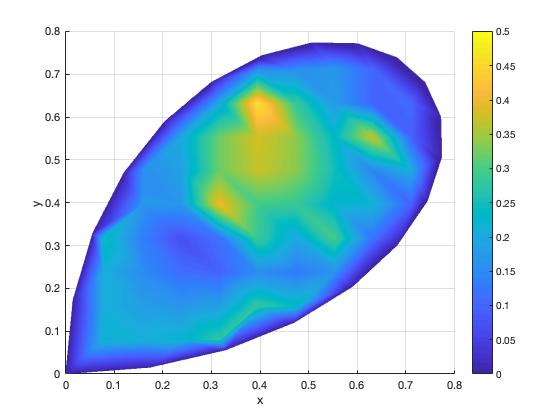}}
\subfloat[$R=2.5$]{\includegraphics[width=0.2\textwidth]{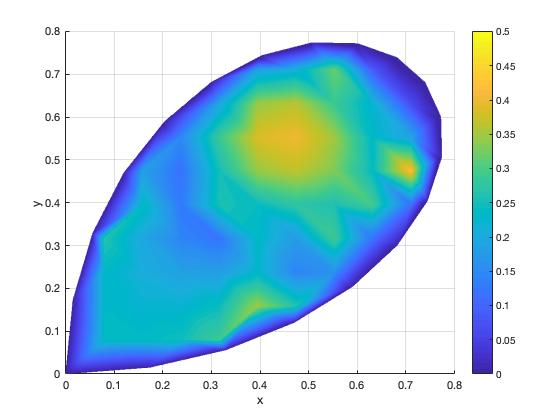}}
\subfloat[$R=3$]{\includegraphics[width=0.2\textwidth]{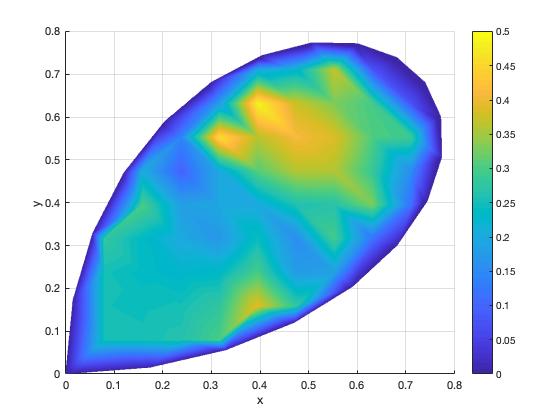}}
\caption{Relative errors $E_2(x_1,x_2)$ for different $c$ and $R$ of Example \ref{ex2d}(b).}\label{ex2d3}
\end{figure}

\begin{figure}[htbp]
  \centering
\subfloat[Relative error $E_3(t)$ for different $R$]{\includegraphics[width=0.3\textwidth]{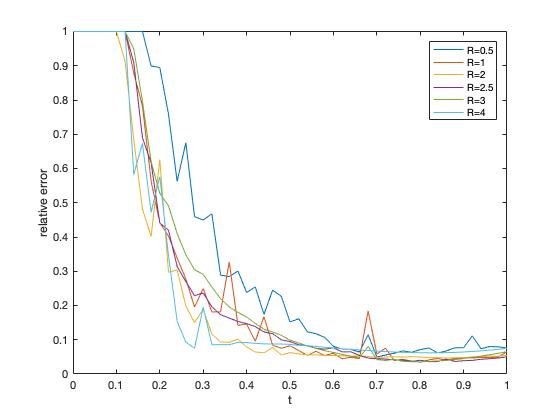}}\quad
\subfloat[Relative error $E_3(t)$ for different $c$]{\includegraphics[width=0.3\textwidth]{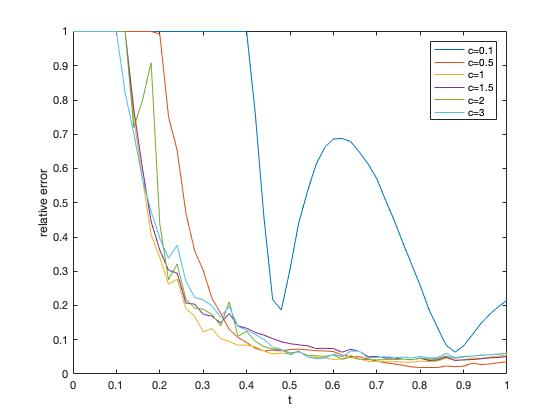}}
\caption{Relative errors  $E_3(t)$ for different $c$ and $R$ of Example \ref{ex2d}(b).}\label{rex2d3}
\end{figure}

Finally, we conducted numerical experiments with a discontinuous initial value function. As shown in Figure \ref{rex2d5}(b), the relative error is generally bigger in this case. However, the numerical results remain acceptable. Figure \ref{ex2d5} illustrates the comparison between the numerical solution for system  \eqref{fp} and approximations for {\bf Problem (C)}.

\begin{figure}[htbp]
  \centering
\subfloat[$E_2(x_1,x_2)$]{\includegraphics[width=0.3\textwidth]{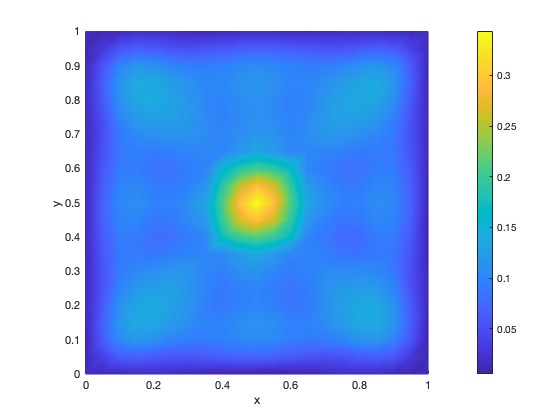}}\quad
\subfloat[$E_3(t)$]{\includegraphics[width=0.3\textwidth]{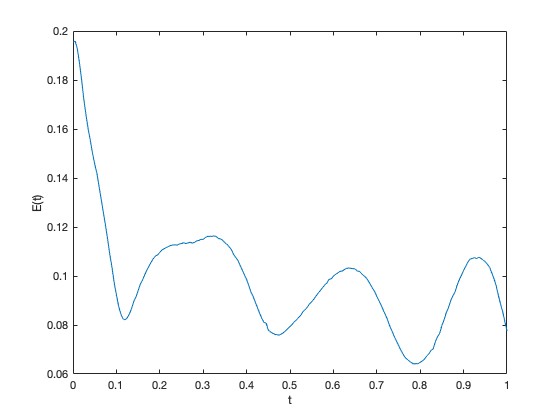}}
\caption{Relative errors of Example \ref{ex2d}(c) with $\delta=7\%$}\label{rex2d5}
\end{figure}


\begin{figure}[htbp]
  \centering
      \captionsetup[subfloat]{labelsep=none,format=plain,labelformat=empty}
\subfloat{\includegraphics[width=0.3\textwidth]{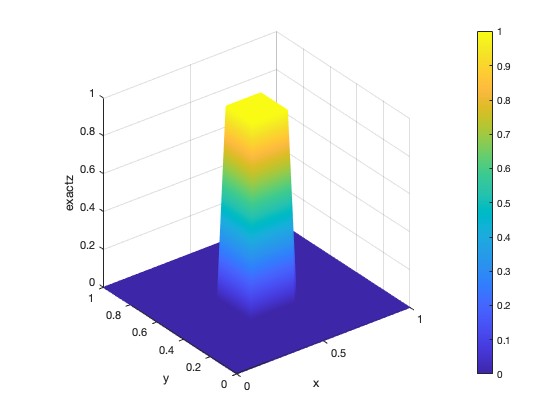}}\ 
\subfloat{\includegraphics[width=0.3\textwidth]{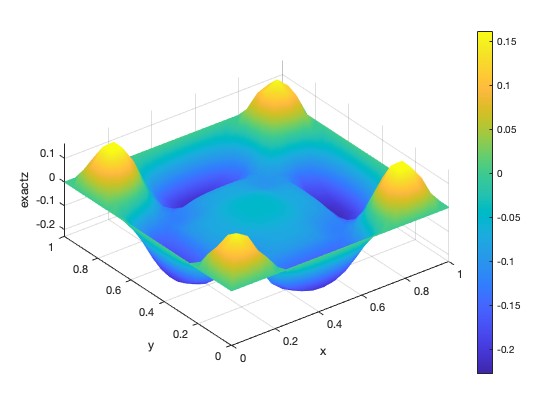}}\ 
\subfloat{\includegraphics[width=0.3\textwidth]{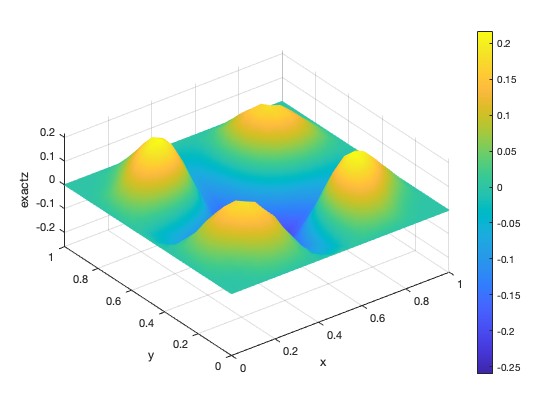}}\\
\subfloat[$t=0$]{\includegraphics[width=0.3\textwidth]{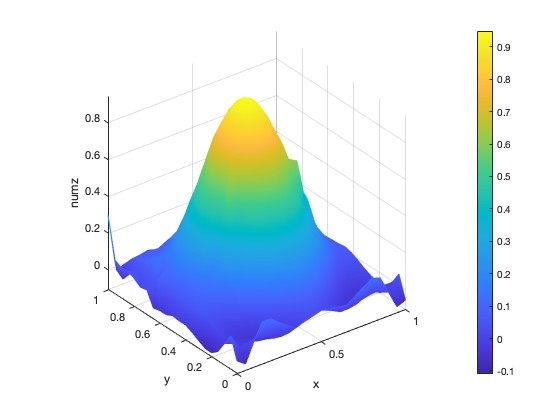}}\ 
\subfloat[$t=0.5$]{\includegraphics[width=0.3\textwidth]{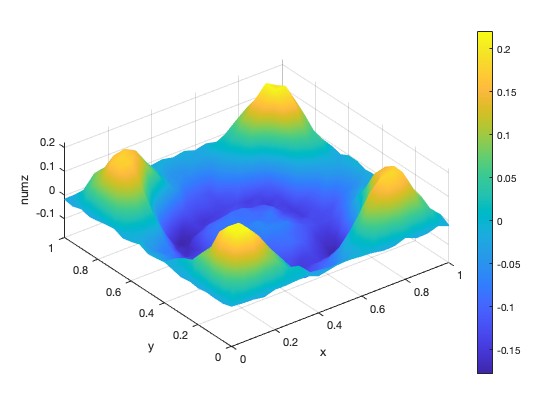}}\ 
\subfloat[$t=1$]{\includegraphics[width=0.3\textwidth]{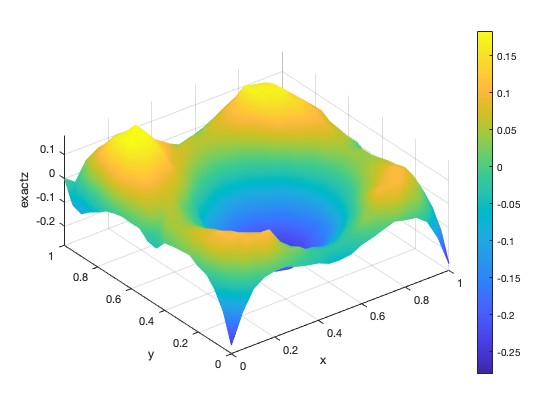}}
\caption{Numerical solution of system \eqref{fp} (above) and approximations of {\bf Problem (C)} (below) for $\delta=7\%$ of Example \ref{ex2d}(e) when $t=0$ , $t=0.5$ and $t=1$.}\label{ex2d5}
\end{figure}

\section{Conclusion}

We have studied an inverse Cauchy problem for the stochastic hyperbolic equation without an initial condition. An observability estimate is established based on a pointwise Carleman identity. Under a suitable a-priori assumption, we construct an approximation using Tikhonov regularization strategy, and prove that the minimizer of the proposed functional exist uniquely, as well as the convergence rate of the minimizer to the exact solution. An easily implementable numerical algorithm is provided by means of kernel-based learning theory and mild solution to the stochastic hyperbolic equation. Since the exact solution for the stochastic hyperbolic equation can not be expressed exactly, we solving the initial-boundary value problem by virtue of  kernel-based learning theory and radial basis functions to obtain the boundary observation data. Numerical experiments demonstrated that the proposed algorithm performed well, even when the initial condition is discontinuous and the spatial domain is irregular. However, as the Green's function becomes more irregular when the dimension of spatial domain exceeds 3, more technique should be supplement to enhance the numerical algorithm. Furthermore, it is more challenging to address other inverse problems related to stochastic hyperbolic equation, such as the inverse potential problem and the inverse source problem, and to explore additional properties of these equations. We plan to study these issues in future work.

\section*{Acknowledgement}
This work is partially supported by the NSFC (No. 12071061), the Science Fund for Distinguished Young Scholars of Sichuan Province (No. 2022JDJQ0035).

\end{document}